\documentclass{amsproc}

\numberwithin{equation}{section}

\usepackage{latexsym}
\usepackage[dvips]{graphicx}
\usepackage{amsmath}
\usepackage{amsfonts}
\usepackage{amssymb}
\usepackage{bm}

\usepackage[english]{babel}

\usepackage{hyperref}

\newcommand{\be}{\begin{equation}}
\newcommand{\ee}{\end{equation}}
\newcommand{\bes}{\begin{equation} \begin{split}}

\newcommand{\smskip}{\vspace{2mm}} 

\newcommand{\ef}[1]{\, #1}

\newcommand{\bz}{\vec{z}}
\newcommand{\cA}{{\mathcal A}}
\newcommand{\cB}{{\mathcal B}}
\newcommand{\ssm}{\smallsetminus}
\def\smfrac#1#2{{\textstyle\frac{#1}{#2}}}
\def\smbinom#1#2{{\textstyle\binom{#1}{#2}}}

\newcommand{\cI}{\mathcal{I}}
\newcommand{\cJ}{\mathcal{J}}
\newcommand{\kS}{\mathfrak{S}}

\newcommand{\sss}[1]{\!\scriptstyle{#1}\!}

\newtheorem{theorem}{Theorem}

\newtheorem{lemma}{Lemma}
\newtheorem{prop}{Proposition}

\title
[Doubly-refined enumeration of Alternating Sign Matrices\ldots]
{Doubly-refined enumeration of Alternating Sign Matrices\\
  and determinants of 2-staircase Schur functions}

\author{Philippe Biane}
\address{Philippe Biane, \newline
\rule{18pt}{0pt}CNRS, IGM --- Universit\'e Paris-Est \newline
\rule{18pt}{0pt}77454 Marne-la-Vallée Cedex 2, FRANCE}
\email{bian{}e@un{}iv-m{}lv.fr}

\author{Luigi Cantini}
\address{Luigi Cantini, \newline
\rule{18pt}{0pt}LPTM, and CNRS --- Universit\'e de
Cergy-Pontoise \newline
\rule{18pt}{0pt}95302 Cergy-Pontoise Cedex, FRANCE}
\email{lu{}igi.cant{}ini@u-c{}ergy.fr}

\author{Andrea Sportiello}
\address{Andrea Sportiello, \newline
\rule{18pt}{0pt}LIPN, and CNRS --- Universit\'e Paris-Nord \newline
\rule{18pt}{0pt}93430 Villetaneuse Cedex, FRANCE}
\email{Andre{}a.Sporti{}ello@lip{}n.univ-par{}is13.fr}

\date{\today}

\begin{document}

\begin{abstract}
We prove a determinantal identity concerning Schur functions for
2-staircase diagrams $\lambda=(\ell n+\ell',\ell n,\ell
(n-1)+\ell',\ell (n-1),\cdots,\ell+\ell',\ell,\ell',0)$.  When
$\ell=1$ and $\ell'=0$ these functions are related to the partition
function of the $6$-vertex model at the combinatorial point and hence
to enumerations of Alternating Sign Matrices.
A consequence of our result is an identity concerning the
doubly-refined enumerations of Alternating Sign Matrices.
\newline
\begin{center}
{\bf \today}
\end{center}
\end{abstract}




\maketitle

\section{Introduction} 

\subsection{Alternating Sign Matrices}

An \emph{alternating sign matrix} (ASM) is a square matrix with
entries in $\{-1,0,+1\}$, such that on each line and on each column,
if one forgets the 0's, the $+1$'s and $-1$'s alternate, and the sum
of each line and each column is equal to 1. It is a famous
combinatorial result that the number of such matrices of size $n$ is
\be
\label{eq.AnEnum}
A_n=\prod_{j=0}^{n-1} 
\frac{(3j+1)!}{(n+j)!}
=
1,2,7,42,429,\ldots
\ee
After having been a conjecture for several years \cite{cit.MRR-asm},
this was first proven by Zeilberger in \cite{cit.Z}, and a simpler
proof was given by Kuperberg \cite{cit.K}, using a connection with the
6-Vertex Model of statistical mechanics, and an appropriate
multivariate extension of the mere counting function $A_n$.
A vivid account can be found in~\cite{cit.B}. 

It follows easily from the definition that an alternating sign matrix
has exactly one $+1$ on its first (and last) row (and column). Thus we
have a sensible four-variable refined statistics, for these four
positions in $\{1,\ldots,n\}^4$, together with their projections on a smaller
number of variables. The dihedral symmetry of the square leaves with a
single one-variable statistics (showing a round formula),
and with two doubly-refined statistics: one, $\cA^n_{ij}$, for the
first and last row (or the rotated case), and one, $\cB^n_{ij}$, for
the first row and column (or the three rotated cases), see
fig.~\ref{fig.asm}, left.
\begin{figure}
\setlength{\unitlength}{13pt}
\begin{picture}(12.8,12.8)
\put(0,0){\includegraphics[scale=1.3]{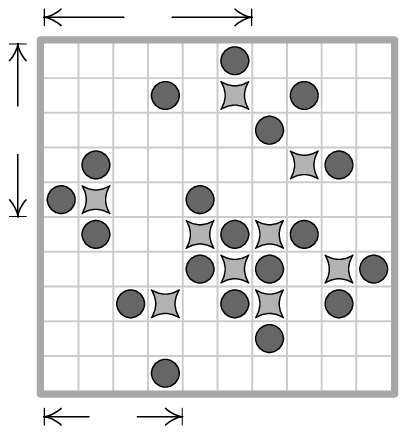}}
\put(4.32,12){$i$}
\put(3.32,0.5){$j$}
\put(0.5,8.72){$k$}
\end{picture}
\quad
\begin{picture}(12.8,12.8)
\put(0,0){\includegraphics[scale=1.3]{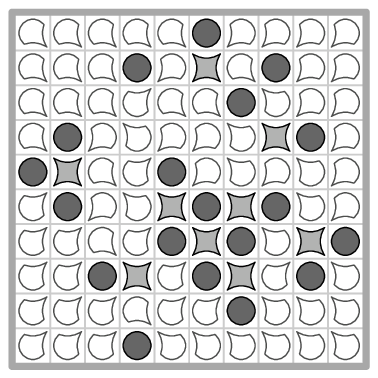}}
\end{picture}
\caption{\label{fig.asm}Left: a typical alternating sign matrix of
  size $n=10$ (empty cells, disks and diamonds stand respectively for
  $0$, $+1$ and $-1$ entries).  This matrix contributes to the
  statistics $\cA^n_{ij}$ and $\cB^n_{ik}$, with $(i,j,k)=(6,4,5)$.
  Right: empty cells are replaced by scale-shaped tiles,
  as to produce a valid tiling (i.e., concavities of neighbouring arcs
  do match). The direction of the tip specifies if
  the cell is of type NW, NE, SE or SW.}
\end{figure}
Matrices $\cA^n$ for $n=1,2,3,4,5$ are given by
\begin{align*}
\cA^1
&=
\begin{pmatrix}
1
\end{pmatrix}
;
&
\cA^2
&=
\begin{pmatrix}
0&1\\
1&0
\end{pmatrix}
;
&
\cA^3
&=
\begin{pmatrix}
0&1&1\\
1&1&1\\
1&1&0
\end{pmatrix}
;
\end{align*}
\begin{align*}
\cA^4
&=
\begin{pmatrix}
0&2&3&2\\
2&4&5&3\\
3&5&4&2\\
2&3&2&0
\end{pmatrix}
;
&
\cA^5
&=
\left(\!
\begin{array}{ccccc}
\sss{0} & \sss{7} & \sss{14} & \sss{14} & \sss{7} \\
\sss{7} & \sss{21} & \sss{33} & \sss{30} & \sss{14} \\
\sss{14} & \sss{33} & \sss{41} & \sss{33} & \sss{14} \\
\sss{14} & \sss{30} & \sss{33} & \sss{21} & \sss{7} \\
\sss{7} & \sss{14} & \sss{14} & \sss{7} & \sss{0}
\end{array}
\!\right)
.
\end{align*}
Of course, by definition $\sum_{i,j} \mathcal{A}^n_{ij} = A_n$, 
i.e.~$1,2,7,42,429,\ldots$ for the cases above.
A simple recursion implies that the sum along the first (and last)
row (and column) gives $A_{n-1}$, 
i.e.~$1,1,2,7,42,\ldots$, and that the bottom-left and top-right
entries are $A_{n-2}$, 
i.e.~$1,1,1,2,7,\ldots$ These simple identities are \emph{linear}.
There exists also \emph{quadratic} relations, of Pl\"ucker nature, relating 
these doubly-refined enumerations to $A_n$
and the (singly-)refined enumerations
(see e.g.~\cite{cit.strog, cit.Colomo2p}).

Evaluate now the \emph{determinant} of these matrices:
\begin{equation*}
\begin{split}
&
\text{det}({\mathcal A^2})=-1 = -1^{-1},\qquad
\text{det}({\mathcal A^3})=1=2^0,
\\
& \qquad
\text{det}({\mathcal A^4})= -7=-7^1, \qquad
\text{det}({\mathcal A^5})= 1764=42^2, \quad \ldots
\end{split}
\end{equation*}
This small numerics suggests a
relation 
that we prove in this paper:
\begin{theorem}\label{th1}
\be
\det (\cA^n)=
(-A_{n-1})^{n-3}
\ef.
\ee
\end{theorem}
\noindent
This relation is \emph{non-linear}. Its degree is not fixed, nor
bounded. What is fixed is what we could call ``co-degree'', namely the
system size, minus the degree (in analogy to the definition of
co-dimension of a subspace). Relations of this different nature seem
to be a novelty for the subject at hand.

Our proof of the theorem above will result as corollary of a
much more general result on certain Schur functions.  To see why these two
topics are connected, we have to revert to Kuperberg solution of the
Alternating Sign Matrix conjecture.

\subsection{ASM, the 6-Vertex Model and Schur functions}

It follows from the connection with the 6-Vertex Model, that the
generating function for a certain weighted enumeration of alternating
sign matrices is given by a closed determinantal formula.  For
$B=\{B_{ij} \}_{1 \leq i,j \leq n}$ an ASM, if $B_{ij}=0$, say that
$(i,j)$ is a \emph{north-west} (NW) site (resp.~NE, SE, SW) if,
forgetting the zeroes, the next $+1$ element along the same column is
in the north direction, and along the same row is in the west
direction (and analogously for the other three cases) -- see
the right part of fig.~\ref{fig.asm}.
Consider some complex-valued function $\mu_n(B)$ over
$n \times n$ ASMs, and call
\be
Z_{n}
=
\sum_{B} \mu_n(B)
\ee
the corresponding generating function (in statistical mechanics
$\mu(B)$ is a generalized \emph{Gibbs measure} -- an ordinary measure
if it is real-positive and normalized -- and $Z$ is the
\emph{partition function}).


When $\mu_n(B)$ has the following factorized form, parametrized by
$2n+1$ variables $(x_1,\ldots,x_n, y_1,\ldots,y_n, q) = (\vec{x},\vec{y},q)$,
\begin{subequations}
\label{eqs.6vw}
\begin{align}
\mu_n(B; \vec{x},\vec{y},q)
&= 
\prod_{1 \leq i,j \leq n} w_{i,j}(B)
\ef;
\\
\label{eq.6vw_b}
w_{i,j}(B)
&=
\left\{
\begin{array}{ll}
(q-q^{-1}) \sqrt{\rule{0pt}{8pt}x_i y_j} & B_{ij} = \pm 1; \\
q^{-1} x_i - q y_j             & B_{ij}=0, \quad \textrm{$(i,j)$ is NW or SE;} \\
-x_i + y_j                     & B_{ij}=0, \quad \textrm{$(i,j)$ is NE or SW;}
\end{array}
\right.
\end{align}
\end{subequations}
integrability methods, and a recursion due to Korepin \cite{korepin}, 
allowed Izergin \cite{izergin} to establish a determinantal 
expression for the generating function $Z_{n}(\vec{x},\vec{y},q) =
\sum_B \mu_n(B; \vec{x},\vec{y},q)$.
In particular, this function is symmetric under $\kS_n \times \kS_n$
acting on row- and column-parameters $x_i$ and $y_j$.

The evaluation of $A_n$ is recovered if we set
$q=\exp(\frac{2 \pi i}{3})$, $x_i=q^{-1}$ for all $i$ and $y_j=q$ for 
all $j$, as in this case the local weights $w_{i,j}$ become all equal
to $i \sqrt{3}$, regardless from $B$, and thus $\mu(B)$ becomes
constant (i.e., the \emph{uniform measure}, up to an overall factor).

Later on it has been recognized \cite{cit.strog, cit.okada} that the
value $q=\exp(\frac{2 \pi i}{3})$ (sometimes called the
\emph{combinatorial point}) has a special combinatorial property:
$Z_{n}(\vec{x},\vec{y},q)$ becomes fully symmetric under $\kS_{2n}$
(acting on the $2n$-uple of $q x_i$'s and $q^{-1} y_j$'s together),
more precisely it is proportional to the Schur function associated to the Young
diagram $\lambda_n = (n-1,n-1,n-2,n-2,\ldots,1,1,0,0)$, evaluated on
variables $\{ q x_1, \ldots, q x_n, q^{-1} y_1, \ldots, q^{-1} y_n \}$
(see figure \ref{fig.YDnll}, left, for a picture of this Young
diagram).
\begin{figure}
\setlength{\unitlength}{20pt}
\begin{picture}(12.2,5.6)(-1.4,0)
\put(-1.4,2.5){$2n$}
\put(-0.8,2.5){$\left\{ \rule{0pt}{54pt} \right.$}
\put(0,0){\includegraphics[scale=2]{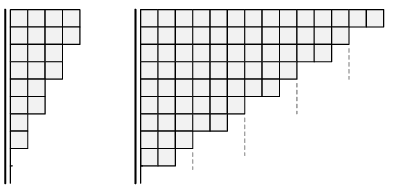}}
\put(5.6,1.6){$\underbrace{\rule{16pt}{0pt}}_{\ell'}$}
\put(5.6,.9){$\underbrace{\rule{29pt}{0pt}}_{\ell}$}
\end{picture}
\caption{\label{fig.YDnll}Left: the Young diagram $\lambda_n$, for
  $n=5$. Right: the Young diagram $\lambda_{n,\ell,\ell'}$, for
  $n=5$, $\ell=3$ and $\ell'=2$.}
\end{figure}
%
One consequence is that we have
\begin{equation}
\label{An}
A_n=
3^{-\binom{n}{2}}
s_{\lambda_n}(1,1,\ldots,1),
\end{equation}
and also the refined enumerations introduced above
are related to specializations of this Schur 
function, in which some parameters are left as indeterminates.

In particular for the $\cA_{ij}^n$'s,
defining the generating function
\be
\cA_n(u,v) 
= 
\sum_{1\leq i,j\leq n}
\cA_{ij}^n
u^{i-1} v^{n-j}
\ef;
\ee 
one finds
\begin{equation}
\label{double}
\cA_n(u,v) 
=
3^{-\binom{n}{2}}
(q^2(q+u)(q+v))^{n-1}
s_{\lambda_n}
\big(
\smfrac{1+qu}{q+u}, \smfrac{1+qv}{q+v}, 1,\ldots,1
\big)
\ef;
\end{equation}
(the rational function $\smfrac{1+qu}{q+u}$ originates from the ratio
of $w_{ij}(B)$ in the two last cases of (\ref{eq.6vw_b})).

A detailed analysis of the double-enumeration formula (\ref{double})
restated in terms of multiple contour integrals, and the proof of a
relation with a double-enumeration formula for totally-symmetric
self-complementary plane partitions in a hexagonal box of size $2n$,
can be found in~\cite{cit.FZJ}.

\subsection{On the determinants of Schur functions}

In this section we state a theorem concerning the determinant of a
matrix whose elements are Schur functions $s_{\lambda_{n}}$.  Not
surprisingly, as these functions are related to ASM enumerations
e.g.\ through equations (\ref{An}) and (\ref{double}), this property
will show up to be the structure behind Theorem \ref{th1}, and
conceivably, it has an interest by itself. For this reason, in this
paper we pursue the task of stating and proving a much wider version
of the forementioned property, than the one that would suffice for
Theorem \ref{th1}. This leads us to introduce a wider family of Young
diagrams.

We define the \emph{2-staircase diagram}
$\lambda_{n,\ell,\ell'}$, for $n \geq 1$, $0 \leq \ell' \leq \ell$, 
as
\be
\label{eq.ladderNL}
\lambda_{n,\ell,\ell'}=\big(
(n-1)\ell+\ell', (n-1)\ell,
(n-2)\ell+\ell', (n-2)\ell,
\ldots, \ell', 0
\big) 
\ee
i.e.\ $(\lambda_{n,\ell,\ell'})_{2j-1}=(n-j)\ell+\ell'$ and
$(\lambda_{n,\ell,\ell'})_{2j}=(n-j)\ell$ (see figure \ref{fig.YDnll},
right).  We call the associated Schur polynomial,
$s_{\lambda_{n,\ell,\ell'}}$, a \emph{$2$-staircase Schur function}.

The name comes from the fact that this family of diagrams generalizes
the well-known family of \emph{staircase diagrams} $\mu_{n,\ell}$
\be
\label{eq.staircasedef}
\mu_{n,\ell} = \big(
(n-1) \ell, (n-2) \ell,\ldots,\ell,0
\big)
\ef.
\ee
The Schur functions $s_{\lambda_n}$ are thus particular cases of
$2$-staircase Schur functions, corresponding to $\ell=1$ and
$\ell'=0$.

The polynomials $s_{\lambda_{n,\ell,\ell'}}$ have been considered
recently by Alain Lascoux. In particular, in \cite[Lemma 13]{cit.Lpf}
they are shown to coincide with the specialization at
$q=\exp(\frac{2\pi i}{\ell + 2})$ of a certain natural extension of
\emph{Gaudin functions}.

In an apparently unrelated context we see the appearence of the
polynomials $s_{\lambda_{n,\ell,\ell'}}$, for $\ell'=0$ only.  This
context, analysed by Paul Zinn-Justin in \cite{cit.pzj_hs}, is the study
of the solution of the $q$KZ equation related to the spin $\ell/2$
representation of the quantum affine algebra
$U_q(\widehat{\mathfrak{sl}(2)})$ with
$q=\exp(\frac{2 \pi i}{\ell + 2})$. 
It is shown that, by taking the scalar product of the
solution of the
$q$KZ equation with a natural reference
state, one obtains $s_{\lambda_{n,\ell,0}}$. 

As anticipated, our Theorem \ref{th1} will be a corollary of the
following result, of independent interest, which exhibits a remarkable
factorization of a determinant of $2$-staircase Schur functions:

\begin{theorem}
\label{th2}
Let $N=\ell (n-1)+\ell' +1$. Let $\{ x_i, y_i\}_{1 \leq i \leq N}$ be
indeterminates, let $f(\vec{z},w_1,w_2)$ stand for
$f(z_1,\ldots,z_{2n-2},w_1,w_2)$, and, for an ordered $N$-uple $\vec
x= (x_1,x_2,\dots,x_N)$, let $\Delta(\vec x)=\prod_{i<j}(x_i-x_j)$
denote the usual Vandermonde determinant.
Then
\be
\label{eq1}
\begin{split}
&
\det
\left(
s_{\lambda_{n,\ell,\ell'}}(\vec{z},x_i,y_j)
\right)_{1\leq i,j\leq N}
\\ & 
\qquad 
=
c(n,\ell,\ell')
\,
\Delta(\vec{x})
\Delta(\vec{y})
\;
\bigg(
\prod_{i=1}^{2n-2} z_i^{\ell'(\ell+1)}
\bigg)
\;
s_{\mu_{2n-2,\ell+1}}^{\ell}(\vec{z})
\;
s_{\lambda_{n-1,\ell,\ell'}}^{\ell(n-2)+\ell'-1}(\vec{z})
\ef.
\end{split}
\ee
The quantity $c(n,\ell,\ell')$ is valued in $\{0,\pm 1\}$. More
precisely,
\be
\label{eq.cTh2}
c(n,\ell,\ell')
=
\left\{
\begin{array}{ll}
(-1)^{(n-1) \binom{\ell+1}{2} + \binom{\ell'+1}{2}}
& 
n=1 \textrm{\quad or \quad} \gcd(\ell+2,\ell'+1)=1
\\
0
&
n>1 \textrm{\quad and\quad} \gcd(\ell+2,\ell'+1) \neq 1
\end{array}
\right.
\ee
\end{theorem}
\noindent
Remark that, as well known, the staircase Schur function
$s_{\mu_{2n-2,\ell+1}}$ can be further factorized.
Let us recall the definition of the (bivariate
homogeneous) Chebyshev polynomials (of the second kind)
\be
\label{def-U}
U_h(x,y) = \frac{x^{h+1}-y^{h+1}}{x-y}
 = x^h + x^{h-1} y + \cdots + y^h
\ef.
\ee
One can write (cf.\ equation (\ref{eq.staircase}))
\be
\label{eq.smu}
s_{\mu_{N,h}}(\vec{z}) = \prod_{1\leq i<j\leq N}U_{h}(z_i,z_j)
\ef.
\ee
\noindent
As Schur functions have several determinant representations (see
Appendix \ref{sec.schur}), the left-hand--side quantity of the theorem is a
``determinant of determinants'', a structure in linear algebra that is
sometimes called a \emph{compound determinant}
\cite[ch.~VI]{cit.muir}. As we will see, the theory of compound
determinants will have a crucial role in our proof.

Results in the form of Theorem \ref{th2}, or at least approaches to
quantities as in the left-hand side of equation (\ref{eq1}), already
exist in the literature, although mostly with partitions of
comparatively simpler structure. Cf.~\cite{cit.Lpf}, where also a
general approach is outlined. In particular, equations (23) and (24)
of \cite{cit.Lpf} have a form of striking similarity with our theorem
above, while involving respectively a rectangular partition $r^p
\equiv (r,r,\ldots,r)$ ($p$ times), and the basic $1$-staircase
partition $(r,r-1,r-2,\ldots,1,0)$, and the unnumbered third equation
after Corollary 9 of \cite{cit.Lpf} (for which, however, no
factorization is stated) has a similar structure to what will be the
matrix of our analysis, with the only difference that it presents a
Chebyshev polynomial at the denominator instead that the numerator.

Theorem \ref{th2} is easily seen to hold at $n=1$ and any
$(\ell,\ell')$.  This could seem a good base for an induction. However
we use inductive arguments only for the minor task of determining the
overall constant $c(n,\ell,\ell')$, in section
\ref{ssec.const}. Conversely, in section \ref{ssec.divisehard}
we prove divisibility
results, by a method reminiscent of the ``exhaustion of factors''
described in Krattenthaler's survey \cite{cit.Kdet}. 


Note however that the factors $s_{\lambda_{n-1,\ell,\ell'}}$ are
polynomials of `large' degree, $\ell n(n-1)+\ell' n$, with no
factorizations as long as $\gcd(\ell+2,\ell'+1)=1$ (we give a partial
proof of this statement in Proposition \ref{unfact} below -- a full
proof is not hard to achieve).  Thus, in a sense, the tools we develop
in section \ref{sec.preli} should be regarded as an extension of the
exhaustion of factor method to the case in which we have an infinite
family of determinantal identities, and some of the factors have an
\emph{unbounded} degree, scaling with the size parameter associated to
the family.

Finally, let us add a few words on notations: along the paper, if
$\vec{z}$ is a vector of length $n$ (the length will be clear by the
context), we write $f(\vec{z})$ as a shortcut for
$f(z_1,\ldots,z_{n})$, and
$f(\vec{z},w_1,w_2,\ldots)$ as a shortcut for
$f(z_1,\ldots,z_{n},w_1,w_2,\ldots)$. We also write and
$f(\vec{z}_{\ssm i_1 \cdots i_k},w_1,w_2,\ldots)$ if the variables
$z_{i_1}$, \ldots, $z_{i_k}$ are dropped from the list $(z_1, \ldots,
z_n)$.

\vskip .3cm

The paper is organized as follows.  In section \ref{sec.1to2} we show
how to derive Theorem~\ref{th1} from Theorem~\ref{th2} specialized to
$\ell=1$ and $\ell'=0$.  In section \ref{sec.preli} we present some
preparatory lemmas to the proof of Theorem \ref{th2}, which is
presented in section \ref{sec.pf}.  Appendix \ref{sec.schur} collects
some basic definitions and facts on Schur functions, while in appendix
\ref{app.paul} we introduce an even larger class of staircase Schur
functions, and study some of their properties.

\section{Derivation of Theorem 1 from Theorem 2}
\label{sec.1to2}

\noindent
For a polynomial $f(x,y)$, denote by $f(x,y)|_{[x^i y^j]}$ the
coefficient of the monomial $x^i y^j$.
We first state a simple but useful lemma.
\begin{lemma}
\label{lemma}
Let $P(u,v)$ be a polynomial in two
indeterminates, of degree at most $n-1$ in each variable. 
Call $P = (P(u,v)|_{[u^{i-1} v^{j-1}]})_{1 \leq i,j\leq n}$.
Let $u_i$, $v_j$
be indeterminates, then
\be
\det \big( P(u_i,v_j) \big)_{1 \leq i,j\leq n}
=
\Delta(\vec{u})
\Delta(\vec{v})
\,
\det P
\ef.
\ee
\end{lemma}
\begin{proof}
Call $V(\vec{u})$ the Vandermonde matrix $V_{ij}=u_i^{j-1}$. Then
$\det V(\vec{u}) = \Delta(\vec{u})$, and
the matrix $\big( P(u_i,v_j) \big)_{1 \leq i,j\leq n}$ 
is the product $V(\vec{u})^{\rm T} P \, V(\vec{v})$.
\end{proof}

\smskip
\noindent
This lemma
allows us to state that our Theorem \ref{th2} is equivalent to
\begin{equation}
\label{eq1powers}
\begin{split}
&
\det
\left(
s_{\lambda_{n,\ell,\ell'}}(\vec{z},x,y)|_{[x^i y^j]}
\right)_{0\leq i,j\leq \ell (n-1)+\ell'}=
\\&  \qquad = 
c(n,\ell,\ell')
\;
\bigg( \prod_{i=1}^{2n-2} z_i^{\ell'(\ell+1)} \bigg)
\;
s_{\mu_{2n-2,\ell+1}}^{\ell}(\vec{z})
\;
s_{\lambda_{n-1,\ell,\ell'}}^{\ell(n-2)+\ell'-1}(\vec{z})
\ef,
\end{split}
\end{equation}
(of course, with $c(n,\ell,\ell')$ as in 
(\ref{eq.cTh2})).

Now we proceed to the proof of Theorem \ref{th1}.
One can compute, with $\vec{u}=(u_1,\ldots,u_n)$,
\begin{equation}
\label{vander}
\Delta\big( \big\{ \smfrac{1+qu_i}{q+u_i} \big\} \big)
=
\Delta(\vec{u})
\,
(q^2-1)^{\binom{n}{2}}
\prod_i(q+u_i)^{-(n-1)}.
\end{equation}
It follows from Lemma \ref{lemma}, and equation
(\ref{double}), that 
\be
\begin{split}
&
\Delta(\vec{u}) \Delta(\vec{v})
\det (\cA_{ij}^n)=
\\
&
\quad =
(-1)^{\binom{n}{2}}
\det
\left(\frac{(q^2(q+u_i)(q+v_j))^{n-1}}
{3^{\binom{n}{2}}}
s_{\lambda_{n}}
\big( \smfrac{1+qu_i}{q+u_i}, \smfrac{1+qv_j}{q+v_j}, 1,\ldots,1 \big)
\right)_{1\leq i,j\leq n}
\\
&
\quad =
\left(
\frac{-q^{4}}{3^n}
\right)^{\binom{n}{2}}
\prod_{i=1}^n \big( (q+u_i)(q+v_i)\big) ^{n-1}
\det \left(
s_{\lambda_{n}}
\big( \smfrac{1+qu_i}{q+u_i}, \smfrac{1+qv_j}{q+v_j}, 1,\ldots,1 \big)
\right)_{1\leq i,j\leq n}
\ef.
\end{split}
\ee
Using Theorem \ref{th2} with $\ell=1$ and $\ell'=0$ on the determinant
on the right-hand side (with $x_i = \smfrac{1+qu_i}{q+u_i}$ and 
$y_j = \smfrac{1+qv_j}{q+v_j}$),
and then (\ref{vander}), we obtain
\begin{equation}
\label{eqA}
\begin{split}
\Delta(\vec{u}) \Delta(\vec{v})
\det
(\cA_{ij}^n)
&=
\Delta(\vec{u}) \Delta(\vec{v})
\,
(-1)^{n-1 + \binom{n}{2}}
\left( \frac{(q-q^2)^2}{3^n} \right)^{\binom{n}{2}}
\\
&\quad \times
s_{\mu_{2n-2,2}}(1,1,\ldots,1)
\;
s_{\lambda_{n-1}}^{n-3}(1,1,\ldots,1)
\end{split}
\end{equation}
Recognize that $(q-q^2)^2 = -3$.  By the explicit evaluation of a
staircase Schur function, equation (\ref{eq.smu}), we have
\begin{equation}
\label{schur3}
s_{\mu_{2n-2,1}}(1,1,\ldots,1)
=
3^{\binom{2n-2}{2}}
\end{equation}
Theorem \ref{th1} follows from (\ref{An}), (\ref{eqA}),
(\ref{schur3}).
\hfill $\square$

\section{Preliminary results}
\label{sec.preli}

\subsection{On the minor expansion of a sum of matrices}

Consider $k$ $n \times n$ matrices of indeterminates $M_{ij}^{(a)}$,
$1 \leq i,j \leq n$; $1 \leq a \leq k$.
For $I,J \subseteq [n]$, denote by $M_{I,J}$ the restriction of $M$ to
rows in $I$ and columns in $J$.
Denote by $\cI = (I_1, \ldots, I_k)$ an ordered $k$-uple of subsets
$I_a \subseteq [n]$ (possibly empty), forming a partition of $[n]$.
For two such $k$-uples $\cI$ and $\cJ$, say that
they are \emph{compatible}
if $|I_a| = |J_a|$ for all $a=1,\ldots,k$, and write $\cI \sim \cJ$ in
this case.
Denote by $\epsilon(\cI,\cJ)$ the signature of the permutation that
reorders $(I_1, \ldots, I_k)$ into $(J_1, \ldots, J_k)$, with elements
within the blocks in order.
Then we have
\begin{prop}[Minor expansion of a sum of matrices]
\label{prop.minexp}
\be
\det \Big( \sum_{a=1}^k M^{(a)} \Big)
=
\sum_{\substack{ \cI, \cJ \\ \cI \sim \cJ }}
\epsilon(\cI,\cJ)
\prod_{a=1}^k
\det M^{(a)}_{I_a,J_a}
\ef.
\ee
\end{prop}
\begin{proof}
Consider the full expansion of the determinant
\bes
\det \Big( \sum_{a=1}^k M^{(a)} \Big)
&=
\sum_{\sigma \in \kS_n}
\epsilon(\sigma)
\prod_{i=1}^n 
\Big( \sum_{a=1}^k M^{(a)}_{i\, \sigma(i)} \Big)
\\
&=
\sum_{\sigma \in \kS_n}
\sum_{b \in [k]^n}
\epsilon(\sigma)
\prod_{i=1}^n 
M^{(b(i))}_{i\, \sigma(i)}
\end{split}
\ee
Associate to each pair $(\sigma,b)$ in the linear combination above, a
pair $(\cI,\cJ)$ of compatible partitions, through
\begin{align}
I_a 
&=
\{ i \; : \; b(i) = a 
\}
\ef;
&
J_a 
&=
\{ j \; : \; b(\sigma^{-1}(j)) = a 
\}
\ef.
\end{align}
So $\cI$ is determined by $b$ alone, and all the permutations $\sigma$
producing the same $\cJ$ can be written as the ``canonical''
permutation $\tau$,
that
reorders $(I_1, \ldots, I_k)$ into $(J_1, \ldots, J_k)$ with elements
within the blocks in order, acting from the left on a permutation 
$\rho = \prod_a \rho_a \in \kS_{I_1} \times \cdots \times \kS_{I_k}$.
The signature factorizes, 
$\epsilon(\sigma) = \epsilon(\tau) \prod_a \epsilon(\rho_a)$,
and $\epsilon(\tau) = \epsilon(\cI,\cJ)$ by definition, thus
\bes
\det \Big( \sum_{a=1}^k M^{(a)} \Big)
&=
\sum_{\substack{ \cI, \cJ \\ \cI \sim \cJ }}
\epsilon(\cI,\cJ)
\prod_a
\sum_{\rho_a \in \kS_{I_a}}
\epsilon(\rho_a)
\prod_{i\in I_a}
M^{(a)}_{i\; \tau \circ \rho_a(i)}
\end{split}
\ee
For each index $a$, the sum over the permutations $\rho_a$ produces
the appropriate determinant of the minor.
\end{proof}

\subsection{Bazin-Reiss-Picquet Theorem}

In this section we recall
the Bazin-Reiss-Picquet 
Theorem~\cite[pg.\ 193-195, \S 202-204]{cit.muir}.

Take a triplet of integers $m \geq n \geq p \geq 0$.
Call $S_{n,p}$ the
set of subsets of $[n]$, of cardinality $p$
(thus $|S_{n,p}| = \binom{n}{p}$). 
For a set $I \in S_{n,p}$, write $I=\{i_1,\ldots,i_p\}$
for the ordered list of elements.

Consider the $m \times n$ matrices of indeterminates $A$ and $B$, 
and the $m \times (m-n)$ matrix of indeterminates $C$. Write $(X|Y)$
for the matrix resulting from taking all the columns of $X$, followed
by all the columns of $Y$.

For a pair $(I,J) \in S_{n,p} \times S_{n,p}$ define $M^{I,J}$ as the
matrix
\be
M^{I,J}_{h,k} = 
\left\{
\begin{array}{ll}
A_{h,k}         &  k \leq n, \ k \not\in I; \\
B_{h,j_{\ell}}  &  k=i_{\ell}; \\
C_{h,k-n}       &  n < k \leq m; 
\end{array}
\right.
\ee
(that is, replace the columns $I$ of $(A|C)$ with the columns $J$ of
$B$, in order). Define $D_{I,J}=\det M^{I,J}$. Choose a total
ordering of $S_{n,p}$, and construct the matrix $D = \big( D_{I,J}
\big)_{I,J \in S_{n,p}}$, of dimension
$\binom{n}{p}$.
Then the compound determinant $\det D$ does not depend on the chosen
ordering, and has the following factorization property:

\begin{theorem}[Bazin-Reiss-Picquet]
\label{thm.bazin}
\be
\det D = 
\det (A|C)^{\binom{n-1}{p}} 
\;
\det (B|C)^{\binom{n-1}{p-1}}
\ef.
\ee
\end{theorem}

\subsection{A divisibility corollary}
\label{ssec.divicoro}

A corollary of the Bazin-Reiss-Picquet Theorem is a divisibility
result for a special family of determinants. Take $m \geq n \geq k
\geq 0$.  Consider $m$ indeterminates $z_i$, $n$ indeterminates
$y_j$, and $2nk$ indeterminates $u_i^a$, $v_i^a$, with $1 \leq i \leq
n$ and $1 \leq a \leq k$ ($u_i^a$, $v_i^a$ may possibly be elements in
the polynomial ring $R(z,y)$).
Take $m$ polynomial functions $f_j(x)$, and introduce the associated
\emph{Slater determinant}, that is, the totally-antisymmetric polynomial
\be
P(\vec{x}) = 
P(x_1,\ldots,x_m)
=
\det \big( f_j(x_i) \big)_{1 \leq i,j \leq m}
\ef.
\ee
A typical example could be a shifted Vandermonde,
$P(x_1,\ldots,x_m) = \Delta_{\lambda}(x_1,\ldots,x_m)$ for
$\lambda$ a partition of length $m$
(see appendix \ref{sec.schur}).

Then we have
\begin{prop}
\label{prop.bazincoro}
The polynomial~\,$
\det \Big( \sum_{a=1}^k u_i^a v_j^a \;
P(\vec{z}_{\ssm i}, y_j)
\Big)_{1 \leq i,j \leq n}
$
is divisible by the polynomial
$\big( P(\vec{z}) \big)^{n-k}$.
\end{prop}
\begin{proof}
Apply the formula for the minor expansion of a sum of matrices,
Proposition~\ref{prop.minexp}, to get
\bes
&
\det \Big( \sum_{a=1}^k u_i^a v_j^a 
P(\vec{z}_{\ssm i}, y_j)
\Big)_{1 \leq i,j \leq n}
\\
& 
\qquad
=
\sum_{\substack{ \cI, \cJ \\ \cI \sim \cJ }}
\epsilon(\cI,\cJ)
\!\!
\prod_{\substack{ 1 \leq a \leq k \\ i \in I_a}} 
\!\!\!
u_i^a
\prod_{\substack{ 1 \leq a \leq k \\ j \in J_a}} 
\!\!\!
v_j^a
\;
\prod_{a=1}^k
\det 
\big(
P(\vec{z}_{\ssm i}, y_j)
\big)_{i \in I_a, \; j \in J_a}
\ef.
\end{split}
\ee
Now apply the Bazin-Reiss-Picquet Theorem to each of the determinants,
with $(m,n,p) \to (m,|I_{\alpha}|,1)$, and get
\be
\det 
\big(
P(\vec{z}_{\ssm i}, y_j)
\big)_{i \in I_a, \; j \in J_a}
=
P(\vec{z})^{|I_a|-1}
P(\vec{z}_{\ssm I_a}, \vec{y}_{\ssm (J_a)^c})
\ef.
\ee
Thus we have
\bes
&
\det \Big( \sum_{a=1}^k u_i^a v_j^a 
P(\vec{z}_{\ssm i}, y_j)
\Big)_{1 \leq i,j \leq n}
\\
& 
\qquad
=
P(\vec{z})^{n-k}
\sum_{\substack{ \cI, \cJ \\ \cI \sim \cJ }}
\epsilon(\cI,\cJ)
\!\!
\prod_{\substack{ 1 \leq a \leq k \\ i \in I_a}} 
\!\!\!
u_i^a
\prod_{\substack{ 1 \leq a \leq k \\ j \in J_a}} 
\!\!\!
v_j^a
\;
\prod_{a=1}^k
P(\vec{z}_{\ssm I_a}, \vec{y}_{\ssm (J_a)^c})
\ef;
\end{split}
\ee
and the quantity in the sum on the right-hand side is a polynomial.
\end{proof}

\subsection{Vanishing and recursion properties of $2$-staircase Schur
  functions}\label{ssec.paul} 

Here we gather some relevant facts about the family of 2-staircase
Schur functions $s_{\lambda_{n,\ell,\ell'}}(\vec{z})$ introduced
in~(\ref{eq.ladderNL}). In this section we use $q$ as a synonym of
$\exp(\frac{2 \pi i}{\ell+2})$.

\begin{prop}[\emph{wheel condition}]
\label{thm.spinwheel}
For distinct $g$, $h$ and $k$ in $\{0,\ldots,\ell+1\}$, and 
distinct $i$, $j$ and $m$
in $\{1,\ldots,2 n\}$, 
\be
\label{wheel-eq123}
s_{\lambda_{n,\ell,\ell'}}(\vec{z}_{\ssm i j m}, q^{g} w, q^{h} w, q^{k} w)
=
0
\ef.
\ee
\end{prop}

\begin{prop}[\emph{recursion relation}]
\label{thm.spinkore}
For $k$ in $\{1,\ldots,\ell+1\}$, and $i$, $j$
in $\{1,\ldots,2 n\}$, distinct,
\be
\label{rec-schur}
s_{\lambda_{n,\ell,\ell'}}(\vec{z}_{\ssm i j}, w, q^{k} w)
=w^{\ell'}
U_{\ell'}(1, q^{k})
\prod_{\substack{
1 \leq m \leq 2 n \\
m \neq i,j }}
\frac{U_{\ell+1}(z_m, w)}{z_m - q^{k} w}
~s_{\lambda_{n-1,\ell,\ell'}}(\vec{z}_{\ssm i j})
\ef.
\ee
\end{prop}
\noindent
Propositions \ref{thm.spinwheel} and \ref{thm.spinkore} are
occurrences, already known in the literature
(cf.\ e.g.\ \cite[Thm.~4]{cit.pzj_hs}), of vanishing conditions (and
related recursion properties) within a broad family, for which the
name ``wheel condition'' is often used. There has been a recent
interest in the investigation of the structure of the corresponding
ideals, in the ring of symmetric polynomials (see
e.g.\ \cite{cit.FJMM2, cit.FJMM}).

We prove the propositions above in Appendix~\ref{app.paul}. More
precisely, in the appendix we generalize $2$-staircase Schur functions
to the \emph{$m$-staircase} case, and prove the appropriate
generalizations of the propositions above, together with some further
properties of potential future interest.

Notice that, if $\gcd(\ell'+1,\ell+2)=g>1$, then there exists some 
$1 \leq k \leq \ell+1$ such that $q^k$ is a root of $U_{\ell'}(1,x)$ 
(e.g., $k=(\ell+2)/g$).
Then it follows from equation (\ref{rec-schur}) that
$s_{\lambda_{n,\ell,\ell'}}$ vanishes if $z_i=q^k z_j$, i.e.\ it is
divisible by $z_i - q^k z_j$. On the
contrary, if $\gcd(\ell'+1,\ell+2)=1$, one has
the following proposition


\begin{prop}
\label{unfact}
Suppose $\gcd(\ell'+1,\ell+2)=1$ and $n\geq 2$, then
$s_{\lambda_{n,\ell,\ell'}}$ has no factors of the form  $(z_i-\eta z_j)$,
for any $1\leq i,j\leq 2n$ and $\eta \in\mathbb{C}$.
\end{prop}
\begin{proof}
We prove the statement by induction on $n$.  The case $n=2$ is done by
direct inspection of $s_{\lambda_{n,\ell,\ell'}}$\;\footnote{%
E.g.,
  realize that, for $z_1-\eta z_2$ to divide the Schur function, it
  should divide the shifted Vandermonde at numerator, with a higher
  power w.r.t.\ the ordinary Vandermonde at denominator. The case
  $\eta=1$ is easily ruled out (even if we further specialize $z_3=z$,
  $z_4=0$, we obtain $s_{\lambda_{2,\ell,\ell'}}(z,z,z,0) =
  z^{2(\ell+\ell')} (\ell+2)(\ell'+1)(\ell-\ell'+1)/2$,
which is not identically zero as we have $\ell \geq 0$ and $0 \leq
\ell' \leq \ell$). For $\eta \neq 1$ we
  can have no simplifications with the Vandermonde at denominator, and
  it suffices to analyse the shifted Vandermonde, which gives
\[
\Delta_{\lambda_{2,\ell,\ell'}}(z,\eta z,0,1) =
  z^{\ell+\ell'+3} \big( ((\eta z)^{\ell+2}-1) (\eta^{\ell'+1}-1) - 
  ((\eta z)^{\ell'+1}-1) (\eta^{\ell+2}-1) \big)
\ef.
\]
Again, this is not identically zero, as, for the $\gcd$ hypothesis,
$\eta^{\ell'+1}-1$ and $\eta^{\ell+2}-1$ cannot vanish
simultaneously.}.  Now suppose the statement true up to $n-1$ and
assume that there exists $i,j \in \{1,\ldots, 2n\}$ and $\eta \in
\mathbb{C}$ such that $(z_i-\eta z_j)$ divides
$s_{\lambda_{n,\ell,\ell'}}$. Then take $k$ and $h$ distinct indices
in $\{1, \ldots, 2n\} \ssm \{i,j\}$ (note that we need $n \geq 2$ at
this point), and specialize $s_{\lambda_{n,\ell,\ell'}}|_{z_k=q
  z_h}$. The linear term $z_i-\eta z_j$ must divide also the
specialized polynomial, and, using the recursion relation of
Proposition \ref{thm.spinkore}, it must divide the corresponding
right-hand--side expression for (\ref{rec-schur}). However, this
expression is non-zero for the other variables $z_m$ being generic
(because the only potentially dangerous factor, $U_{\ell'}(1, q^{k})$,
may vanish only if $\gcd(\ell'+1,\ell+2)>1$), and the factors of the
form $z_k^{\ell'}$, and $U_{\ell+1}(z_m, z_k)$,
for $m \neq k,h$, do not contain $z_i-\eta z_j$ as a factor. Thus
$z_i-\eta z_j$ must divide $s_{\lambda_{n-1,\ell,\ell'}}$,
this being in contrast with the inductive assumption.
\end{proof}

\section{Proof of Theorem 2}
\label{sec.pf}

\noindent
As outlined in the introduction, our strategy for proving Theorem
\ref{th2} will be as follows: let us call $\psi_{n,\ell,\ell'}(z,x,y)$
the left-hand side of (\ref{eq1}); first we identify several
polynomial factors of $\psi_{n,\ell,\ell'}(z,x,y)$; then we show that
these factors are relatively prime and that their product exhausts the
degree of $\psi_{n,\ell,\ell'}(z,x,y)$; finally, we determine the
overall constant factor. As in the previous subsection, also in this
section we set $q=e^{\frac{2 \pi i}{\ell+2}}$.

\subsection{Polynomial factors of $\psi_{n,\ell,\ell'}(\vec z,\vec x,\vec y)$}
\label{ssec.divisehard}

We start by identifying a polynomial factor of
$\psi_{n,\ell,\ell'}(\vec z,\vec x,\vec y)$ whose
factorization
involves only monomials and binomials.  By virtue of Lemma
\ref{lemma}, we have that $\psi_{n,\ell,\ell'}(\vec z,\vec x,\vec y)$
is divisible by $\Delta(\vec{x})$ and $\Delta(\vec{y})$. Since the
degree of $\psi_{n,\ell,\ell'}$ in each variable $x_i$ or $y_i$
separately is $(n-1)\ell+\ell'$, which is the same as the degree of
$\Delta(\vec{x}) \Delta(\vec{y})$, the quotient is a polynomial of
degree zero in $x_i$ and $y_j$ (namely, it is the determinant of the
matrix of coefficients in $x$ and $y$ of $s_{\lambda_{n,\ell,\ell'}}
(\vec{z},x,y)$).
Call $Q_{n,\ell,\ell'}(\vec z)$ the resulting quotient
\be
\label{def-Q}
Q_{n,\ell,\ell'}(\vec z) =
\frac{\psi_{n,\ell,\ell'}(\vec z,\vec x,\vec y)}{\Delta(\vec{x})\Delta(\vec{y})} 
\ee
We work out immediately the case of Theorem \ref{th2}
corresponding to the second case of equation (\ref{eq.cTh2})
\begin{prop}
\label{gcd}
If $\gcd(\ell'+1, \ell+2) > 1$ and $n\geq 2$, 
then $Q_{n,\ell,\ell'}(\vec z)=0$
\end{prop}
\begin{proof}
Say $\gcd(\ell'+1, \ell+2) = g > 1$. It follows that the polynomials
$U_{\ell'}(1,x)$ and $U_{\ell+1}(1,x)$ have a common root $q^k$, for
$k=(\ell+2)/g$.  We can exploit the fact that $Q$, defined in equation
(\ref{def-Q}) as a rational function of the $z$, $x$ and $y$'s, is
actually independent from the $x$ and $y$'s. In particular, we can
choose $x_1=q^k z_1$ (and leave $x_2, \ldots, x_n, y_1, \ldots, y_n$
generic).  Consider the matrix 
$M_{ij}=s_{\lambda_{n,\ell,\ell'}}(\vec{z},x_i,y_j)$, whose
determinant is $\psi_{n,\ell,\ell'}$.
By applying the recursion relation of Proposition \ref{thm.spinkore}
we see that the row corresponding to $x_1$ vanishes identically. On
the other side, as the remaining $x$ and $y$ variables are generic,
the Vandermonde factors are non-zero. As a consequence,
$Q_{n,\ell,\ell'}(\vec z)=0$.
\end{proof}

\noindent
We proceed to find other factors of $Q_{n,\ell,\ell'}$, for the
relevant case of equation (\ref{eq.cTh2}).
\begin{prop}
\label{powers}
For $n \geq 2$, $s_{\mu_{2n-2,\ell+1}}^{\ell}(\vec{z})
\Big( \prod_{i=1}^{2n-2} z_i^{\ell'(\ell+1)} \Big)$ divides
$Q_{n,\ell,\ell'}(\vec z)$. 
\end{prop}
\begin{proof}
Note that $Q_{n,\ell,\ell'}(\vec z)$ is symmetric in the $z_i$'s (as
they enter only as simultaneous arguments of Schur functions).  So,
given the factorized form of $s_{\mu}$, equation (\ref{eq.smu}), it
suffices to prove that $Q$ is divided by
$z_1^{\ell'(\ell+1)}\prod_{m=2}^{2n-2}U_{\ell+1}^\ell(z_1,z_m)$.
Using the independence from $\vec x$ and $\vec y$ of equation
(\ref{def-Q}), we can choose to substitute $x_i=q^i z_1$ for $1\leq i
\leq \ell+1$, and leave generic the other $x_j$'s, and all the $y_j$'s
(we have a sufficient number of $x$'s since $(n-1)\ell+\ell'+1 \geq
\ell+1$ for $n\geq 2$).

By applying the recursion relation of Proposition \ref{thm.spinkore}
to the matrix elements $M_{ij}$,
the first $\ell+1$ rows of $M$ are simplified. Consider the matrix
$\widetilde{M}$, that coincides with $M$ on rows $i > \ell+1$, and
otherwise is given by
\be
\label{eq.456765}
\begin{split}
\widetilde{M}_{ij}
&=
\bigg(
z_1^{\ell'}U_{\ell'}(1, q^{i})\prod_{m=2}^{2 n-2}
\frac{U_{\ell+1}(z_m, z_1)}{z_m - q^i z_1}
\bigg)
\bigg(
\frac{U_{\ell+1}(y_j, z_1)}{y_j - x_i}
\,
s_{\lambda_{n-1,\ell,\ell'}}(\vec{z}_{\ssm 1},y_j)
\bigg)
\end{split}
\ee
This matrix is a version of $M$ in which we do \emph{not} replace $x_i
\to q^i z_1$ for all the occurrences of $x_i$ in $M_{ij}$, but only
for a subset. That is, we just have the property, for $1 \leq i \leq \ell+1$,
\be
M_{ij}
=
\left.
\widetilde{M}_{ij} \right|_{x_i=q^i z_1}
\ef,
\ee
and thus $\det M = ( \det \widetilde{M} )|_{x_i=q^i z_1}$. We
constructed $\widetilde{M}$ instead of $M$ with specific intentions: the
two factors in parenthesis in (\ref{eq.456765}) are separately
polynomials after replacing $x_i=q^i z_1$ (and, before the replacing,
they are divided at most by $y_j-x_i$); the factor on the left does
not depend on index $j$ (so it can be extracted from the $i$-th row of
$\widetilde{M}$ when evaluating the determinant); finally, the dependence from
$i$ in the second factor is all due to $x_i$ (so that the $i$-th and
$i'$-th row of $M$ are the same vector of functions, with different
$x$ argument, i.e.~$\det \widetilde{M}$ is at sight divisible by
$\Delta(x_1,\ldots,x_{\ell+1})$).

The factors extracted from the rows give
\be
\prod_{i=1}^{\ell+1}
\Big(
z_1^{\ell'}
U_{\ell'}(1, q^{i})
\prod_{2 \leq m \leq 2 n-2}
\frac{U_{\ell+1}(z_m, z_1)}{z_m - q^i z_1}
\Big)
\ef,
\ee
that is, with some simplifications
(including 
$\prod_{i=1}^{\ell+1}U_{\ell'}(1,q^i) = 1$ if $\gcd(\ell+2,\ell'+1)=1$
and 0 otherwise),
\be
\label{eq.4662435654}
z_1^{\ell'(\ell+1)}
\prod_{m=2}^{2n-2}U_{\ell+1}^\ell(z_1,z_m)
\ef.
\ee
The divisibility of $\det \widetilde{M}$ by
$\Delta(x_1,\ldots,x_{\ell+1})$ implies that $\det \widetilde{M} /
\Delta(x_1,\ldots,x_N)$ has no factors $x_i - x_{i'}$ at the
denominator with $1 \leq i < i' \leq \ell+1$, and thus no pure powers
of $z_1$ at the denominator from the Vandermonde, after the
replacement $x_i=q^i z_1$ (indeed, all the potential factors at the
denominator have the form $q^{i} z_1 -x_j$, with $j > \ell+1$, and
$y_j-q^i z_1$, with $j \leq \ell+1$), thus they do not affect the
claimed factor in (\ref{eq.4662435654}). This completes the proof.
\end{proof}

Now we complete the exhaustion of factors, by proving the following
weaker form of Theorem \ref{th2}
\begin{prop}
\label{lem.grossodiv}
For $n \geq 2$ and $\ell \geq 1$
we have
\be
\label{weaker}
Q_{n,\ell,\ell'}(\vec z) = c(n,\ell,\ell')
\bigg( \prod_{i=1}^{2n-2} z_i^{\ell'(\ell+1)} \bigg)
\;
s_{\mu_{2n-2,\ell+1}}^{\ell}(\vec{z})
\;
s_{\lambda_{n-1,\ell,\ell'}}^{\ell(n-2)+\ell'-1}(\vec{z})
\ef,
\ee
for some numerical constant $c(n,\ell,\ell')$.
\end{prop}

\begin{proof}
As a consequence of Proposition \ref{gcd}, our claim is trivially true
if $\gcd(\ell'+1,\ell+2)>1$,  as the constant in such a case is $0$.
Therefore it remains to analyse the case $\gcd(\ell'+1,\ell+2)=1$.

We can again exploit the invariance in $x$ and $y$ of 
$Q_{n,\ell,\ell'}(\vec z)$ from equation (\ref{def-Q}),
in order to evaluate $\psi_{n,\ell,\ell'}(\vec z, \vec x,\vec y)$ at a specially
simpler set of values $x$ and $y$.
Our choice is to leave the $y_j$'s generic,
and specialize $x_i = q^{k_i} z_{m_i}$, 
for all the indices $i=1,\ldots, \ell (n-1) +\ell'+ 1$, and 
$\{ (k_i,m_i) \}$ being a whatever ordered subset of distinct pairs,
of cardinality $\ell (n-1) +\ell'+ 1$, in the set of all valid pairs 
$\{1, \ldots, \ell+1 \} \times \{1, \ldots, 2n-2\}$ (the difference of
cardinality, $(\ell+2)(n-1)-\ell'-1$, is always positive in our range of
interest $\ell \geq 1$, $0 \leq \ell' \leq \ell$, $n \geq 2$).
Using Theorem \ref{thm.spinkore}
we have
\bes
M_{ij} 
&=
s_{\lambda_{n,\ell,\ell'}} (\vec{z},x_i=q^{k_i} z_{m_i},y_j)
=\\
&=z_{m_i}^{\ell'}U_{\ell'}(1, q^{k_i})
\frac{U_{\ell+1}(y_j, z_{m_i})}{y_j - q^{k_i} z_{m_i}}
\prod_{\substack{
1 \leq r \leq 2 n-2\\r\neq m_i}}
\frac{U_{\ell+1}(z_r, z_{m_i})}{z_r - q^{k_i} z_{m_i}}
s_{\lambda_{n-1,\ell,\ell'}}(\vec{z}_{\ssm m_i},y_j)
\ef.
\end{split}
\ee
Let us adopt the representation (\ref{eq.schurpol}) for the Schur
polynomial (as the ratio of shifted Vandermonde over Vandermonde), to
get
\bes\label{Mij2}
M_{ij} 
&=
\frac{z_{m_i}^{\ell'}U_{\ell'}(1, q^{k_i})}{\Delta(\vec{z}_{\ssm m_i},y_j) }
\frac{U_{\ell+1}(y_j, z_{m_i})}{y_j - q^{k_i} z_{m_i}}
\prod_{\substack{
1 \leq r \leq 2 n-2\\r\neq m_i}}
\frac{U_{\ell+1}(z_r, z_{m_i})}{z_r - q^{k_i} z_{m_i}}
\ 
\Delta_{\lambda_{n-1,\ell,\ell'}} (\vec{z}_{\ssm m_i},y_j)
\\
&=
\frac{z_{m_i}^{\ell'}U_{\ell'}(1, q^{k_i})}{\Delta(\vec{z}) }
(-1)^{m_i+1}
\bigg(
\prod_{r \neq m_i}
\frac{(z_r - z_{m_i})
U_{\ell}(z_r,z_{m_i})}{z_r - q^{k_i}z_{m_i}}
\bigg)
\bigg(
\prod_r \frac{1}{y_j - z_r}
\bigg)
\\
& \qquad \times
\frac{(y_j - z_{m_i})
U_{\ell}^{(0,k_i)}(y_j,z_{m_i})}{y_j - q^{k_i}z_{m_i}}
\ 
\Delta_{\lambda_{n-1,\ell,\ell'}} (\vec{z}_{\ssm m_i},y_j)
\\
&=
\frac{z_{m_i}^{\ell'}U_{\ell'}(1, q^{k_i})}{\Delta(\vec{z}) }
\bigg(
(-1)^{m_i+1}
\prod_{r \neq m_i}
U_{\ell+1}(z_r, q^{k_i} z_{m_i})
\bigg)
\bigg(
\prod_r \frac{1}{y_j-z_r}
\bigg)
\\
& \qquad \times
U_{\ell+1}(y_j, q^{k_i} z_{m_i})
\ 
\Delta_{\lambda_{n-1,\ell,\ell'}} (\vec{z}_{\ssm m_i},y_j)
\ef, 
\end{split} \ee
where in the last equality we made use of the relation
\be\label{rel-U}
\frac{U_{\ell+1}(x,q^hy)}{x - q^ky}=
\prod_{\substack{0 \leq i \leq \ell+1 \\ i \neq h,k}}
(x - q^i y)
=\frac{U_{\ell+1}(x,q^ky)}{x - q^hy}
\ef.
\ee
In the last expression of equation (\ref{Mij2}), we recognize five
factors: a factor
independent on $i$ and $j$, one depending on $i$ alone, one depending
on $j$ alone, and one depending on both $i$ and $j$, which is composed
of $U_{\ell+1}(y_j, q^{k_i} z_{m_i})$, that is 
a homogeneous polynomial in $y_j$ and $z_{m_i}$ of degree $\ell+1$,
and a shifted Vandermonde.
The first three factors are easily extracted when evaluating $\det M$,
so we can write
\be
\label{detMsimpl}
\det M=
\frac{A(\vec z,\vec y)}{B(\vec z,\vec y)
\Delta(\vec{z})^{N}}
\ 
\det \widehat{M}
\,
\ee
with
\begin{align}
\widehat{M}_{ij} 
&=
-U_{\ell+1}(y_j, q^{k_i} z_{m_i})
\Delta_{\lambda_{n-1,\ell,\ell'}} (\vec{z}_{\ssm m_i},y_j)
\ef;
\\
A(\vec z,\vec y)
&=
\prod_i(-1)^{m_i+1}z_{m_i}^{\ell'}
U_{\ell'}(1,q^{k_i}) 
\prod_{\substack{ 
1 \leq i \leq N \\
1 \leq r \leq 2n-2 \\ r \neq m_i }}
\!\!\!
U_{\ell+1}(z_r, q^{k_i} z_{m_i})
\ef;
\\
B(\vec z,\vec y)
&= 
\prod_{\substack{ 1 \leq j \leq N \\
1 \leq r \leq 2n-2 }} 
\! (y_j-z_r)
\ef.
\end{align}
We now substitute the expression of equation (\ref{detMsimpl}) in
(\ref{def-Q}), where we also 
replace
\be
\Delta(\vec{x})
\quad \longrightarrow \quad
\Delta(q^{k_1} z_{m_1}, q^{k_2} z_{m_2}, \ldots)
\ef,
\ee
which leads to
\bes
Q_{n,\ell,\ell'}(\vec z)
&=
\frac{A(\vec z,\vec y)}{B(\vec z,\vec y)}
\frac{1}{\Delta(\vec{y}) \Delta(q^{k_1} z_{m_1}, q^{k_2} z_{m_2},
  \ldots) \Delta(\vec{z})^N}
\; \det \widehat{M}
\ef;
\end{split}
\label{eq.Qandnow2}
\ee
Now, the matrix $\widehat{M}$ is in a form suitable for application of
Proposition \ref{prop.bazincoro}, the divisibility result discussed in
Section \ref{ssec.divicoro}, with $k=\ell + 2$ and, for $0 \leq a \leq
\ell+1$, $u_i^a v_j^a$ being the coefficient of the monomial $y_j^a
z_{m_i}^{\ell+1-a}$ in the expansion of $U_{\ell+1}(y_j, q^{k_i}
z_{m_i})$.

As a consequence we get that
$\Delta_{\lambda_{n-1,\ell,\ell'}}^{N - (\ell+2)}(\vec{z})$
divides $\det \widehat{M}$, and the exponent
$N-(\ell+2) = \ell(n-1)+\ell'+1-(\ell+2)=\ell(n-2)+\ell'-1$ is exactly
the desired one from the statement of Proposition \ref{lem.grossodiv}
(and Theorem \ref{th2}). So we can write
\be
\det \widehat{M} 
=
\Delta_{\lambda_{n-1,\ell,\ell'}}^{\ell(n-2)+\ell'-1} (\vec{z})
\;
R(z,y)
\ee
for $R$ a polynomial.
We thus recognize in (\ref{eq.Qandnow2})
\be
\label{eq.Qandnow3}
Q_{n,\ell,\ell'}(\vec z)
=
s_{\lambda_{n-1,\ell,\ell'}}^{\ell(n-2)+\ell'-1} (\vec{z})
\;
\frac{A(z,y) R(z,y)}{B(z,y)\Delta(\vec{y}) \Delta(q^{k_1} z_{m_1}, q^{k_2} z_{m_2},
  \ldots)\Delta(\vec{z})^{\ell+2}}
\ef.
\ee
Now, as $\gcd(\ell'+1,\ell+2)=1$, we obtain two consequences from
Proposition \ref{unfact}. First, observing that the denominator 
in (\ref{eq.Qandnow3}) is completely factorized into linear terms (of
the form $y_i-z_j$, or $z_i - q^k z_j$),
$s_{\lambda_{n-1,\ell,\ell'}}(\vec{z})$ 
cannot be divided by any of these factors,
therefore it follows from equation (\ref{eq.Qandnow3}) that
$s_{\lambda_{n-1,\ell,\ell'}}^{\ell(n-2)+\ell'-1} (\vec{z})$ must
divide $Q_{n,\ell,\ell'}(\vec z)$. 

Furthermore, we know from Proposition \ref{powers} that
$s_{\mu_{2n-2,\ell+1}}^{\ell}(\vec{z})\prod_{i=1}^{2n-2}
z_i^{\ell'(\ell+1)}$ divides $Q_{n,\ell,\ell'}(\vec z)$.  Also this
polynomial is factorized into linear terms, of the form $z_i$ or $z_i
- q^k z_j$, thus it is relatively prime with
$s_{\lambda_{n-1,\ell,\ell'}}^{\ell(n-2)+\ell'-1}$.  This shows that
Proposition \ref{lem.grossodiv} holds, for $c(n,\ell,\ell')$ a
polynomial. However, all the involved functions are homogeneous
polynomials, and it is easily determined that $c(n,\ell,\ell')$ has
degree~0, thus it is a constant.
\end{proof}

\subsection{Determine the constant $c(n,\ell,\ell')$}
\label{ssec.const}

We can evaluate directly the constant for $n=1$, which is
$c(1,\ell,\ell')=(-1)^{\binom{\ell'+1}{2}}$, and we know that, for
$n\geq 2$ and $\gcd(\ell+2,\ell'+1)>1$, $c(n,\ell,\ell')=0$.
In the rest of this section we will complete the proof of the expression
(\ref{eq.cTh2}), for the remaining case
$n\geq 2$ and $\gcd(\ell+2,\ell'+1)=1$.
This is done by induction in $n$, 
i.e.\ we will prove that, for $(n,\ell,\ell')$ as above,
\be
\frac{c(n,\ell,\ell')}{c(n-1,\ell,\ell')}
=
(-1)^{\binom{\ell+1}{2}}
\ef.
\ee
Now that we only have to determine the constant, we have the freedom
of choosing simpler values also for the $z_k$'s, besides that for the
$x_i$'s and the $y_j$'s.

First of all, in equation (\ref{def-Q}) 
let us specialize $x_i=q^iz_1$ for
$1\leq i \leq \ell$. In this way we find that the matrix elements
$M_{ij}$ for $1 \leq i \leq \ell$
take the form\footnote{That is, nothing but $\widetilde{M}_{ij}$ in
  (\ref{eq.456765}), under the full replacement $x_i \to q^i z_1$.}
\be
M_{ij}
=
z_1^{\ell'}U_{\ell'}(1, q^{i})
\frac{U_{\ell+1}(y_j, z_1)}{y_j - q^{i} z_1}
\prod_{r=2}^{2n-2}
\frac{U_{\ell+1}(z_r, z_1)}{z_r - q^{i} z_1}
s_{\lambda_{n-1,\ell,\ell'}}(\vec{z}_{\ssm 1},y_j)
\label{eq.237654375}
\ee
As we have done in the proof of Proposition \ref{powers}, when we
compute the determinant of the matrix $M$, for $1\leq i\leq \ell$ we
extract the factor  
\be
\label{fact3}
z_1^{\ell'}U_{\ell'}(1, q^{i})\prod_{r=2}^{2n-2}
\frac{U_{\ell+1}(z_r, z_1)}{z_r - q^i z_1}
\ee
from the $i$-th row, and find
\be
\label{det-tilde}
\det M = F(z_1;\vec z_{\ssm 1}) \det 
M'
\ee
where
\be
F(z_1;\vec z_{\ssm 1})
=
\frac{z_1^{\ell'\ell}}{U_{\ell'}(1,q^{\ell+1})}
\prod_{r=2}^{2n-2}
U_{\ell+1}^{\ell-1}(z_r,z_1)
\;
(z_r-q^{\ell+1}z_1)
\ee
and the matrix $M'$
coincides with $M$ along the last 
$N-\ell$
rows, while each of the first $\ell$ rows is divided by the
factor in equation (\ref{fact3}). 

We now substitute the expression (\ref{det-tilde}) for $\det M$ into
the definition of $Q_{n,\ell,\ell'}(\vec z)$ and then into
equation (\ref{weaker}), taking into accout also the variable
substitutions in the Vandermonde at denominator  
\begin{gather}
\label{eq.09876465}
\Delta(\vec x)
\quad
\longrightarrow  
\quad
z_1^{\binom{\ell}{2}} 
\Delta'(z_1, \vec x_{\ssm 1,\dots \ell})
\\
\Delta'(z_1, \vec x_{\ssm 1,\dots \ell})
:=
\Delta(q, q^2, \ldots, q^\ell) 
\;
\Delta(\vec x_{\ssm 1,\dots \ell})
\prod_{\substack{1 \leq i \leq \ell \\
\ell+1 \leq k \leq N
}}
(q^iz_1-x_k) 
\ef.
\end{gather}
It could be explicitly evaluated, although not needed for our purposes,
that
\be
\Delta(q, q^2, \ldots, q^\ell)^2
=
(-1)^{\binom{\ell+1}{2}}
(q^{-1}-q^{-2})^2
(\ell+2)^{\ell-2}
\ef.
\ee
We obtain
\bes
Q_{n,\ell,\ell'}(\vec z)
&=
\frac{F(z_1;\vec z_{\ssm 1})
\det M'
}
{z_1^{\binom{\ell}{2}} 
\Delta'(z_1, \vec x_{\ssm 1,\dots \ell})
\Delta(\vec y)}
\\ 
&=
c(n,\ell,\ell')\prod_{i=1}^{2n-2} z_i^{\ell'(\ell+1)}
\;
s_{\mu_{2n-2,\ell+1}}^{\ell}(\vec{z})
\;
s_{\lambda_{n-1,\ell,\ell'}}^{\ell(n-2)+\ell'-1}(\vec{z})
\ef,
\end{split}
\ee
We eliminate the factors appearing on both sides of the previous
equation and we obtain
\bes
\label{partial-rel}
&\frac{\det M'
}{U_{\ell'}(1,q^{\ell+1})
\Delta'(z_1, \vec x_{\ssm 1,\dots \ell})
\Delta(\vec y)}
= c(n,\ell,\ell')
z_1^{\binom{\ell}{2}+\ell'}
\\ 
& \qquad \times
\prod_{i=2}^{2n-2} z_i^{\ell'(\ell+1)}
\;
\prod_{r=2}^{2n-2}\frac{U_{\ell+1}(z_r,
q^{\ell+1}z_1)}{z_r-z_1}
s_{\mu_{2n-3,\ell+1}}^{\ell}(\vec{z}_{\ssm 1})
\;
s_{\lambda_{n-1,\ell,\ell'}}^{\ell(n-2)+\ell'-1}(\vec{z})
\ef.
\end{split}
\ee
Note that, among other things, we have eliminated some factors
$z_r-q^{\ell+1}z_1$ 
on both sides, a simplification that
allows us to set $z_2=q^{\ell+1}z_1$. Furthermore, we choose to
specialize $y_i=q^i z_1$, for $1\leq i\leq \ell$ (the Vandermonde
factor $\Delta(\vec{y})$ in equation (\ref{partial-rel})
is then to be treated similarly to what is done in
(\ref{eq.09876465}) for $\Delta(\vec x)$).

It is easy to see which simplifications occur on the factorized
right-hand side of equation (\ref{partial-rel})
\begin{align}
\prod_{i=2}^{2n-2} z_i^{\ell'(\ell+1)}
&\rightarrow
q^{\ell'} z_1^{\ell'(\ell+1)}
\prod_{i=3}^{2n-2}
z_i^{\ell'(\ell+1)}
\\ 
\prod_{r=2}^{2n-2}
\frac{U_{\ell+1}(z_r, q^{\ell+1}z_1)}{z_r-z_1}
&\rightarrow
\frac{z_1^\ell(\ell+2)}{q^{-2}-q^{-1}}
\prod_{r=3}^{2n-2}
\frac{U_{\ell+1}(z_r, q^{\ell+1}z_1)}{z_r-z_1}
\\
s_{\mu_{2n-3,\ell+1}}^{\ell}(\vec{z}_{\ssm 1})
&\rightarrow
\prod_{r=3}^{2n-2}U_{\ell+1}^\ell(z_r,q^{\ell+1} z_1)
\;
s_{\mu_{2n-4,\ell+1}}^{\ell}(\vec{z}_{\ssm 1,2})
\\  
s_{\lambda_{n-1,\ell,\ell'}}
(\vec{z})
&\rightarrow
z_1^{\ell'}U_{\ell'}(1, q^{\ell+1})
\prod_{r=3}^{2n-2}
\frac{U_{\ell+1}(z_r, z_1)}{z_r - q^{\ell+1} z_1}
s_{\lambda_{n-2,\ell,\ell'}}
(\vec{z}_{\ssm 1,2})
\ef.
\end{align}
Even more drastic simplifications arise
on the left-hand side of equation (\ref{partial-rel}).
For
$i>\ell$ and 
$j\leq \ell$,
the entries
$M'_{ij}$
consist of the Schur polynomials
$s_{\lambda_{n,\ell,\ell'}}$ evaluated at a set of variables including
a triple satisfying the wheel condition (namely,
$z_1$,
$y_j=q^j z_1$ and $z_2=q^{\ell+1}z_1$),
therefore they vanish because of Proposition \ref{thm.spinwheel}.
Similarly, for $i \leq \ell$ and 
$j\leq \ell$,
with the only exception of $i=j$,
$M'_{ij}$ vanishes because of the factor
$\frac{U_{\ell+1}(y_j, z_1)}{y_j - q^{i} z_1} = 
\prod_{1 \leq k \leq \ell+1;\, k\neq i} (y_j - q^k z_1)$
(cf.\ equation (\ref{eq.237654375})).
As a result,
\be
\label{eq.64826474}
\det M' = 
\Big( \prod_{i=1}^{\ell} M'_{ii} \Big)\;
\det M'_{\{\ell+1,\ldots,N\}, \{\ell+1,\ldots,N\}}
\ee
The diagonal factors $M'_{ii}$ read
\be
M'_{ii}
=
\frac{z_1^{\ell+\ell'}
  (\ell+2)q^{\ell i}U_{\ell'}(q^{\ell+1},q^i)}{1-q^{-i}}\prod_{r=3}^{2n-2}
\frac{U_{\ell+1}(z_r,q^{\ell+1}z_1)}{z_r-q^iz_1}s_{\lambda_{n-2,\ell,\ell'}}(\vec   
z_{\ssm 1,2}).
\ee
Most importantly, the minor of the matrix $M'$ restricted to the last
$N-\ell$ rows and columns is easily related to the matrix $M$ for the
system of size $n-1$, where the indices of the variables $z_k$ run
from $3$ to $2n-2$, while the indices of the $x_i$'s and $y_j$'s run
from $\ell+1$ to $N = (n-1)\ell+\ell'+1$. More precisely,
$M'_{\ell+i,\ell+j}$, at size $n$ and under the specializations above,
is proportional to $M_{ij}$ at size $n-1$, the proportionality factor
for the pair $(i,j)$ being
\bes
z_1^{\ell'}U_{\ell'}(1, q^{\ell+1})
\bigg(
\prod_{r=3}^{2n-2}
\frac{U_{\ell+1}(z_r, z_1)}{z_r - q^{\ell+1} z_1}
\bigg)
\frac{U_{\ell+1}(x_{\ell+i}, z_1)}{x_{\ell+i} - q^{\ell+1} z_1}
\;
\frac{U_{\ell+1}(y_{\ell+j}, z_1)}{y_{\ell+j} - q^{\ell+1} z_1}
\end{split}
\ee
(the relevant fact is that this quantity factorizes into a term
depending on $x_i$ only, and a term depending on $y_j$ only, these
terms thus factorize in the evaluation of the determinant).
Thus we get
\bes
&
\det M'_{\{\ell+1,\ldots,N\}, \{\ell+1,\ldots,N\}}
=
\bigg[
z_1^{\ell'}U_{\ell'}(1, q^{\ell+1})
\bigg(
\prod_{r=3}^{2n-2}
\frac{U_{\ell+1}(z_r, z_1)}{z_r - q^{\ell+1} z_1}
\bigg)
\bigg]^{N-\ell}
\\
& \qquad \times 
\prod_{i=1}^{N-\ell}
\frac{U_{\ell+1}(x_{\ell+i}, z_1)}{x_{\ell+i} - q^{\ell+1} z_1}
\;
\prod_{j=1}^{N-\ell}
\frac{U_{\ell+1}(y_{\ell+j}, z_1)}{y_{\ell+j} - q^{\ell+1} z_1}
\\
& \qquad \times 
\Delta(x_{\ell+1},\ldots,x_N)
\;
\Delta(y_{\ell+1},\ldots,y_N)
\; 
Q_{n-1,\ell,\ell'}(z_3,\ldots,z_{2n})
\ef.
\end{split}
\label{lhs-partial}
\ee
In this equation we can substitute $Q_{n-1,\ell,\ell'}(\vec z_{\ssm
  1,2})$ with its expression given by equation (\ref{weaker}) -- the
factor $c(n-1,\ell,\ell')$ emerges at this point -- then, we can
replace (\ref{lhs-partial}) in (\ref{partial-rel}), using
(\ref{eq.64826474}).  In this way we reach a fully factorized form on
both sides of equation (\ref{partial-rel}) and erasing the common
factors is reduced to simple algebra\footnote{Useful relations at this
  point are
\begin{align*}
\prod_{i=1}^\ell \frac{q^{\ell i}}{1-q^{-i}}
&=
q^{-2}-q^{-1}
\ef;
&
\prod_{i=1}^\ell U_{\ell'}(q^{\ell+1},q^i) 
&=
\frac{q^{\ell'}}{U_{\ell'}(1,q^{\ell+1})}
\ef. 
\end{align*}
}. At the end, we obtain the recursive relation
\be
c(n,\ell,\ell')= (-1)^{\binom{\ell+1}{2}} c(n-1,\ell,\ell')
\ef,
\ee
as was to be proven.
\hfill $\square$


\appendix
\section{Basic facts on symmetric polynomials}
\label{sec.schur}


\noindent
A \emph{partition} $\lambda$ of length $k$ is a non-increasing
sequence of $k$ non-negative numbers, $\lambda=(\lambda_1\geq
\lambda_2\geq \ldots\geq \lambda_k\geq 0)$.  The number of terms (or
\emph{parts}) $\ell(\lambda)=k$, and the value of the sum
$|\lambda|=\sum_{i=1}^k\lambda_i$, are called respectively the
\emph{length} and the \emph{weight} of the partition.  Seen as a table
of cells (as e.g.\ in figure \ref{fig.YDnll}),
$\lambda$ is often called a \emph{Young diagram}.

Given an ordered $\ell$-uple of indeterminates 
$\bz = \{ z_i \}_{1\leq i \leq \ell}$,
the \emph{Vandermonde determinant} 
$\Delta(\bz)$
is defined as the determinant of the 
$\ell \times \ell$ matrix $V$ with $V_{ij} = z_i^{\ell-j}$.
It is well known that
$
\Delta(\bz) = \prod_{1 \leq i < j \leq \ell}
(z_i - z_j)
$.
For a partition $\lambda$ of length $\ell$ one similarly defines
the \emph{shifted Vandermonde determinant}
$\Delta_{\lambda}(\bz)$
as the determinant of the 
$\ell \times \ell$ matrix $V$ with $V_{ij} = z_i^{\lambda_j+\ell-j}$.
Thus $\Delta(\bz) \equiv 
\Delta_{(0,0,\ldots,0)}(\bz)$.
Then, the \emph{Schur polynomial} associated to $\lambda$ is the
function in $\ell$ indeterminates
\be
\label{eq.schurpol}
s_{\lambda}(\bz)
=
\frac{\Delta_{\lambda}(\bz)}{\Delta(\bz)}
\ef.
\ee
It is indeed a polynomial, it is symmetric in all its variables, and
homogeneous of degree $|\lambda|$.
The Schur functions are at the heart of algebraic combinatorics
\cite{stanley} and enjoy several remarkable properties (see
\cite{macdonald}).
Here we limit ourselves to present the few simple results we need in
the paper, among which a ``splitting formula'':
\begin{prop}
\label{prop.limschur}
Let $\lambda$ and $\mu$ two partitions of lengths respectively $k$ and
$h$, such that $\lambda_{k}\geq \mu_1$.
Call $\nu$
the partition
$\nu=(\lambda_1,\dots,\lambda_{k},\mu_1,\dots,\mu_{h})$. Then 
we have
\be
\label{limit-schur}
\lim_{\epsilon \rightarrow 0}
\frac{s_\nu(z_1, \ldots, z_k,\epsilon y_1, \ldots, \epsilon y_h)}{\epsilon^{|\mu|}}
= s_{\lambda}(z_1, \ldots, z_k)
\,
s_{\mu}(y_1, \ldots, y_h) 
\ef.
\ee
\end{prop}
\noindent
This generalizes the simple property, that $s_{\lambda}(\vec{z})$ has
maximum degree $\lambda_1$, and mimimum degree $\lambda_{\ell(\lambda)}$, in
any of its variables. For the connoisseurs, the proposition can be
easily proven in several ways, for example by using the decomposition
formula for Schur function 
$s_{\alpha}(\vec{x},\vec{y}) =
\sum_{\beta \subseteq \alpha}
s_{\beta}(\vec{x})
s_{\alpha / \beta}(\vec{y})$
(see e.g.\ \cite[eq.\;(5.9)]{macdonald}), and simple properties of
skew Schur functions (that we do not introduce).
Here we provide a more verbose but completely self-contained proof.

\begin{proof}
Using the defining equation (\ref{eq.schurpol}), we are led to study the
behaviour of $\Delta_{\gamma}(\vec z,\epsilon
\vec y)$ as
$\epsilon\rightarrow 0$, for the cases $\gamma=\nu$ (at numerator) and
$\gamma=0$ (at denominator). More generally, consider
$\gamma=(\gamma_1,\ldots,\gamma_{k+h})
\equiv (\alpha_1,\ldots,\alpha_k,\beta_1,\ldots,\beta_h)$.
Recall that $\Delta_{\gamma}(\vec z,\epsilon
\vec y)$ is defined as the
determinant  of the matrix $V_{ij} =
z_i^{\gamma_j+k+h-j}$ for $i\leq k$ and $V_{ij} =
(\epsilon y_{i-k})^{\gamma_j+k+h-j}$ for $i>k$. Consider 
the Laplace expansion of $V$ along the first $k$ rows:
\be
\det V
= \sum_{\substack{I \subseteq [k+h] \\ |I| = k}}
\epsilon(I,[k])
\det V_{[k],I} \det V_{[k]^{c}, I^{c}}
\ef.
\ee
As the summand with index $I$ has an exposed factor
$\epsilon^{\sum_{j \in I^c} (\gamma_j+k+h-j)}$, the term with $I=[k]$ has
a factor $\epsilon^{|\beta| + \binom{h}{2}}$, and all other terms have a
higher power. Thus
\bes
\frac{\Delta_{\gamma}(\vec z,\epsilon \vec y)}
{\epsilon^{|\beta| + \binom{h}{2}}}
&=
\det(z_i^{\alpha_j+k+h-j})_{1 \leq i,j \leq k}
\;
\det(y_i^{\beta_j+k+h-(k+j)})_{1 \leq i,j \leq h}
\;
+ \mathcal{O}(\epsilon)
\\
&=
\Big( \prod_{i=1}^k z_i^h \Big)
\Delta_{\alpha}(\vec{z})
\;
\Delta_{\beta}(\vec{y})
\;
+ \mathcal{O}(\epsilon)
\ef.
\end{split}
\ee
Comparing this equation for $\gamma=\nu$ and $\gamma=0$ allows us to conclude.
\end{proof}

\vspace{2mm}

\noindent
The bivariate homogeneous Chebyshev polynomials of the second kind are
defined as
\be
U_k(x,y) = \frac{x^{k+1}-y^{k+1}}{x-y}
 = x^k + x^{k-1} y + \cdots + y^k
\ef.
\ee
Define the \emph{staircase partition} $\mu_{n,\ell}$ as the
length-$n$ partition
\be
\mu_{n,\ell} = \big(
\ell n - \ell, \ell n - 2 \ell,\ldots,\ell,0
\big)
\ef.
\ee
The associated Schur function
is easily evaluated through (\ref{eq.schurpol})
\bes
s_{\mu_{n,\ell}}(\vec{z}) 
&=
\frac{\Delta_{\mu_{n,\ell}}(\bz)}{\Delta(\bz)}
=
\frac{\Delta(z_1^{\ell+1},\ldots,z_n^{\ell+1})}{\Delta(z_1,\ldots,z_n)}
\\
&=
\prod_{1 \leq i < j \leq n}
\frac{z_i^{\ell+1}-z_j^{\ell+1}}{z_i-z_j}
=
\prod_{1 \leq i < j \leq n}
U_{\ell}(z_i,z_j)
\ef.
\end{split}
\label{eq.staircase}
\ee

\section{Properties of staircase Schur functions}
\label{app.paul}

\noindent
Let us consider three non-negative integers $N$, $m$ and $\ell$, with
$m \geq 1$, and a partition $\lambda'$ of length $m$, with $\lambda'_1
- \lambda'_m \geq \ell$.
We define the partition $\lambda(N,m,\ell,\lambda')$ 
as follows: for $0 \leq i < N$, consider the unique way of writing
$N-i = am+b$, with $a \geq 0$ and $0 \leq b < m$ 
(it is just $a = \lfloor (N-i)/m \rfloor$ and $b \equiv N-i \ ({\rm mod}\ m)$).
Then
\be
\lambda_{N-i} = a m + \lambda'_b
\ef.
\ee
(see fig.\ \ref{fig.lamNLM}).
\begin{figure}
\begin{center}
\setlength{\unitlength}{20pt}
\begin{picture}(8,5.6)(-1.4,-0.2)
\put(0,-.5){\includegraphics[scale=2]{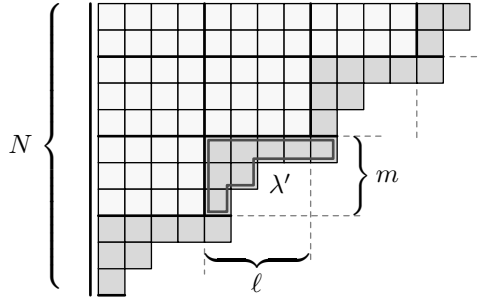}}
\put(-1.4,2.25){$N$}
\put(-0.8,2.25){$\left\{ \rule{0pt}{59pt} \right.$}
\put(2.32,0.4){$\underbrace{\rule{39pt}{0pt}}_{\displaystyle{\ell}}$}
\put(5,1.7){$\left. \rule{0pt}{17pt} \right\} m$}
\put(3.52,1.48){$\lambda'$}
\end{picture}
\end{center}
\caption{\label{fig.lamNLM}An example of partition $\lambda(N,m,\ell,\lambda')$ 
with $N=11$, $m=3$, $\ell=4$ and $\lambda' = (5,2,1)$.}
\end{figure}
We call \emph{$m$-staircase diagrams} such Young diagrams, and
\emph{$m$-staircase Schur functions} the Schur functions in $N$
variables $s_{N,m,\ell,\lambda'}(\vec{z}) \equiv
s_{\lambda(N,m,\ell,\lambda')}(\vec{z})$.  These functions generalize
the ($1$-)-staircase and $2$-staircase functions defined
in (\ref{eq.staircasedef}) and (\ref{eq.ladderNL}), corresponding to
take $m=1$ and $2$ respectively, $\lambda'_m = 0$ and $N$ a multiple
of $m$ ($\lambda'_1 \equiv \ell'$ for $2$-staircase functions).
In this section we set $q = \exp(\frac{2 \pi i}{\ell+m})$.

We say that a symmetric function in $N$ variables $f(z_1,\ldots, z_N)$
satisfies the \emph{$(m,\ell)$-wheel condition} if,
for $I=\{ i_1,\ldots,i_{m+1} \} \subseteq [N]$
and $K=\{ k_1,\ldots,k_{m+1} \} \subseteq [\ell+m]$,
\be
f(z_1,\ldots, z_N)|_{z_{i_a} = q^{k_a} w}
=0
\ef.
\ee
Note that, as we deal with symmetric polynomials, it is not necessary
to take ordered $m$-uples instead of subsets.  We call a
specialization $z_{i_a} = q^{k_a} w$ of the form above a ``wheel
hyperplane''.  Proposition \ref{thm.spinwheel} is the $2$-staircase
function specialization of the following more general proposition.
The proof we produce below is a minor variation of the one presented
in~\cite{cit.pzj_hs} (within the proof of its Theorem 4) for that case.
\begin{prop}
\label{prop1.paulM}
The symmetric function in $N$ variables
$s_{N,m,\ell,\lambda'}(\vec{z})$
satisfies the $(m,\ell)$-wheel condition.
\end{prop}
\begin{proof}
Consider the generic wheel hyperplane $z_{i_a} = q^{k_a} w$ for $i_a
\in I$ and $k_a \in K$ as above.  Call $\lambda=
\lambda(N,m,\ell,\lambda')$ for brevity.  Represent
$s_{N,m,\ell,\lambda'}(\vec{z})$ as a ratio of shifted Vandermonde
over Vandermonde, $\Delta_{\lambda}/\Delta$, as in equation
(\ref{eq.schurpol}).  As, even under the specialization, the variables
$z_i$ are all distinct, the Vandermonde at the denominator is
non-singular, and it suffices to prove that the shifted Vandermonde
vanishes.  The shifted entries of the partition are $\tilde \lambda_i
=\lambda_i + (N-i)$, and writing $i=N-a m-b$, we have
$\tilde{\lambda}_{N-am-b} = (\ell+m)a+b+\lambda'_{m-b}$.  Note in
particular that
\be
\label{eq.7654386}
\tilde{\lambda}_{N-am-b}
\equiv b+\lambda'_{m-b} \quad ({\rm mod}\ \ell+m)
\ee
regardless of $a$.  Consider the matrix $V_{ij} =
z_i^{\tilde{\lambda}_j}$, such that $\Delta_{\lambda} = \det V$.
Call $V'$ the rectangular
minor of $V$, restricted to the $m+1$ rows in $I$, and write
$j=N-a m-b$ as above.  Then, because of equation
(\ref{eq.7654386}),
\be
V'_{ij} =
z_i^{\tilde{\lambda}_{j} } 
=
w^{\tilde{\lambda}_{j} } 
q^{k_i (N - (\ell+m) a_j - b_j - \lambda'_{b_j} ) }
=
w^{\tilde{\lambda}_{j} } q^{N k_i}
q^{-k_i (b_j+\lambda'_{m-b_j})}
\ef.
\ee
As $b + \lambda'_{b}$ for $b \in \{0,\ldots,m-1\}$ takes $m$
distinct values, $V'$ has rank at most $m$, 
while it has $m+1$ rows.
This proves that $\det V = 0$.
\end{proof}

Now we present a generalization of Proposition~\ref{thm.spinkore}.
\begin{prop}
\label{thm.paulM}
For $I=\{ i_1,\ldots,i_{m} \} \subseteq [N]$
and $K=\{ k_1,\ldots,k_{m} \} \subseteq [\ell+m]$,
$s_{N,m,\ell,\lambda'}(\vec{z})$ satisfies the recursion 
\bes
\label{eq.koreLM}
s_{N,m,\ell,\lambda'}&(\vec{z})(\vec{z}_{\ssm I}, q^{k_1} w, \ldots, q^{k_m} w)
=\\
&s_{\lambda'}(q^{k_1} , \ldots, q^{k_m} )w^{|\lambda'|}
\bigg(
\prod_{j \in [N] \ssm I}
\prod_{h \in [\ell+m] \ssm K}
(z_j - q^{h} w)
\bigg)
s_{N-m,m,\ell,\lambda'}(\vec{z}_{\ssm I})
\ef.
\end{split}
\ee
\end{prop}
\begin{proof}
From Proposition \ref{prop1.paulM} it follows that, for
$I$ and $K$ as above,
$s_{N,m,\ell,\lambda'}(\vec{z})$ satisfies the following equation
\be
\label{eq.koreLM2}
s_{N,m,\ell,\lambda'}(\vec{z}_{\ssm I}, q^{k_1} w, \ldots, q^{k_m} w)
=\bigg(
\prod_{j \in [N] \ssm I}
\prod_{h \in [\ell+m] \ssm K}
(z_j - q^{h} w)
\bigg)
F^{(K)}_{N,m,\ell,\lambda'}(\vec{z}_{\ssm I},w)
\ef,
\ee
for some polynomial
$F^{(K)}_{N,m,\ell,\lambda'}(\vec{z}_{\ssm I},w)$. 
Rewrite the equation above in the form
$
\Delta_{\lambda(N,m,\ell,\lambda')}(\vec{z},q^{k_i} w)
=
\Delta(\vec{z},q^{k_i} w)
\prod_{j,h}
(z_j - q^{h} w)
F^{(K)}_{N,m,\ell,\lambda'}
$.
An easy computation on minimal and maximal degree in $w$ for all the
factors in this expression (other than $F^{(K)}$) shows that
$F^{(K)}_{N,m,\ell,\lambda'}(\vec{z}_{\ssm I},w)$ is homogeneous of
degree $|\lambda'|$ in $w$.
Thus, 
\be
F^{(K)}_{N,m,\ell,\lambda'}(\vec{z}_{\ssm I},w)
\equiv
w^{|\lambda'|}
\lim_{v\rightarrow 0}\frac{
F^{(K)}_{N,m,\ell,\lambda'}(\vec{z}_{\ssm I},v)}{v^{|\lambda'|}}
\ee
and, in order to determine this quantity, it suffices
to divide both sides of equation (\ref{eq.koreLM2}) by
$w^{|\lambda'|}$ and take the limit $w\rightarrow 0$. 
Using Lemma \ref{limit-schur} we find for the
left-hand side of equation (\ref{eq.koreLM2})
\be
\lim_{w\rightarrow 0}\frac{s_{N,m,\ell,\lambda'}(\vec{z}_{\ssm I},
  q^{k_1} w, \ldots, q^{k_m} w)}{w^{|\lambda'|}} =  s_{\lambda'}(q^{k_1}, \ldots,
  q^{k_m}) s_{N-m,m,\ell,\lambda'}(\vec{z}_{\ssm I})
\prod_{j\notin I}z_j^\ell  
\ef.
\ee
The factor $\prod_{j\notin I}z_j^\ell$ simplifies with the the same term
appearing on the right-hand side, from the limit of the product of
binomials $z_j - q^{h} w$. 
Therefore we end up with
\be\label{limit2}
{F^{(K)}_{N,m,\ell,\lambda'}(\vec{z}_{\ssm I},w)}
=
{w^{|\lambda'|}}
s_{\lambda'}(q^{k_1} , \ldots, q^{k_m} )
s_{N-m,m,\ell,\lambda'}(\vec{z}_{\ssm I})
\ef.
\ee
\end{proof}

\noindent
For a symmetric polynomial $P(\vec{z})$ in $N$ variables, and $1 \leq
k \leq N$, call $d_k(P)$ the maximum degree of $P$ in (any) $k$
variables simultaneously. In what follows, when the number of
variables is clear, we will use the shortcuts $d \equiv d_1$ and $D
\equiv d_N$, Recall that, for a Schur function $s_{\lambda}(\vec{z})$,
$d_k = \lambda_1 + \ldots + \lambda_k$.

Among the staircase Schur functions considered in the propositions
above, the subclass $\lambda'_1= \ldots = \lambda'_m = 0$
(i.e.\ $\lambda'=\varnothing$) has the further property of being ``of
minimal degree'' among all symmetric functions satisfying the wheel
condition, in various senses involving this set of degrees $d_k$.  The
following proposition describes some of the possible choices. It is a
generalization to the $m$-staircase case of the $m=2$ situation
analysed in \cite[Thm.~4]{cit.pzj_hs}, but, contrarily to Proposition
\ref{prop1.paulM}, the proof technique is substantially different, as
the Lagrange Interpolation argument used in \cite{cit.pzj_hs} is
specific to $m=2$ (with higher values, some degree counting hypothesis
is not met).

Determining the unicity of a function satisfying a precise set of
conditions and degree bounds is often a useful tool when one wants to
``prove that two (families of) functions are the same''. Despite this
could appear as a rare eventuality, this line of reasoning has already
proven valuable in several enumeration problems related to integrable
systems, ranging from the recognition of the Izergin determinant
\cite{izergin}, and its identification as a Schur function
\cite{cit.strog}, up to the ``higher-spin'' cases in
\cite{cit.pzj_hs}.  We report the following result, with the hope that
it may be useful in generalizations of six-vertex and loop models
involving simultaneously both ``higher-spin'' and ``higher rank'',
i.e.\ higher values of $m$ (besides $m=2$) in representations of the
quantum affine algebra $q$-deforming $\mathfrak{sl}(m)$.

\begin{prop}
\label{thm.paulMmin}
Let $N=am+b$, with $a\geq 0$ and $1 \leq b \leq m$.
The symmetric polynomial in $N$ variables
$s_{N,m,\ell}(\vec{z}):=s_{N,m,\ell,\varnothing}(\vec{z})$ 
has $(D,d,d_m)=(D^*,d^*,d_m^*)$, with
\be
(D^*,d^*,d^*_m) = 
\left( 
a \ell \left( \smfrac{m (a-1)}{2} + b \right),
a\ell,
(N-m)\ell
\right)
\ef.
\ee
It is the unique symmetric function satisfying the $(m,\ell)$-wheel
condition (up to multiplication by a scalar), and any of the following
degree conditions:
\begin{enumerate}
\item[$(a)$] $d \leq d^*$ and $D \leq D^*$;
\item[$(b)$]
$d_m \leq d_m^*$;
\item[$(b')$] $d \leq d^*$ and $m$ divides $N$;
\item[$(c)$] $f_{m,\ell}(D,d) \leq f_{m,\ell}(D^*,d^*)$, for
  $f_{m,\ell}(D,d) = \frac{\ell}{m} D + \frac{d(d+\ell)}{2}$.
\end{enumerate}
\end{prop}
\begin{proof}
Clearly $(b')$ is implied by $(b)$ and $d_m \leq m d$ for any
polynomial, so it suffices to concentrate on the three cases $(a)$,
$(b)$ and $(c)$.  Also, clearly a degree condition $d \leq d^*$ alone
would fail unicity, as, if $b<m$, any $s_{N,m,\ell,\lambda'}$, such
that $\lambda'$ has at most $m-b$ non-zero parts and $\lambda'_1 \leq
\ell$, would work.

The fact that the Schur functions above satisfy the claimed wheel
condition has been already proven in Proposition \ref{prop1.paulM},
and the degrees are easily calculated.  So we just have to prove
degree minimality, and unicity.

If we have $a=0$ (i.e.\ $N \leq m$), for arbitrary $m$ and $\ell$, the
statement is trivial because
the wheel condition is empty (there are no wheel
hyperplanes), and indeed $s_{N,m,\ell}(\vec{z})=1$ in this case.

The case $m=1$, and arbitrary $N$ and $\ell$, is also fairly simple. A
polynomial in $N$ variables $P(\vec z)$ satisfies the $(1,\ell)$-wheel
condition if and only if, for all $i < j$ and $1\leq k\leq \ell$, it
is divided by $z_i-q^k z_j$. Therefore the polynomial of minimal
degree satisfying the wheel condition consists of the product of these
factors,
and indeed coincides with $s_{N,1,\ell}(\vec{z}) \equiv
s_{\mu_{N,\ell}}(\vec{z})$.

The proof for generic values of $m$ and $N$, and any of the degree
conditions in the list, is done by a double induction on $N$ and $m$,
using the cases above as a basis.
%
Let us assume the statement to be true up to the value $m-1$, and, for
the value $m$, up to $N-1$ variables.  Then suppose that $P(\vec z)$
is a symmetric
polynomial in $N$ variables, satisfying the
$(m,\ell)$-wheel condition, and with a degree triple $(D,d,d_m)$
satisfying any of the conditions.  We want to show that, up to
rescaling $P(\vec z)$ by a constant factor, 
$P(\vec{z})=s_{N,m,\ell}(\vec{z})$.

We know from Proposition \ref{thm.paulM} that, for $I$ and $K$
appropriate sets (i.e., $I \subseteq [N]$ and $K \subseteq [\ell+m]$,
$|I|=|K|=m$), $P(\vec z)$ satisfies the following equation
\be
\label{eq.koreLM3}
P(\vec{z}_{\ssm I}, q^{k_1} w, \ldots, q^{k_m} w)
=\bigg(
\prod_{\substack{ j \in [N] \ssm I \\ h \in [\ell+m] \ssm K}}
\!\!\!\!
(z_j - q^{h} w)
\bigg)
\;
F^{(K)}_{N,m,\ell}(\vec{z}_{\ssm I},w)
\ef,
\ee 
for some polynomial $F^{(K)}_{N,m,\ell}(\vec{z}_{\ssm I},w)$,
symmetric in the $N-m=(a-1)m+b$ variables $\{z_j \}_{j \not\in I}$,
and satisfying the $(m,\ell)$-wheel
condition on the remaining variables $z_j$.

Call $d(F)$ the maximum degree of $F$ in one variable, seen as a
polynomial in variables $z_j$ only, $d_w(F)$ the degree as a
polynomial in $w$ and $D_w(F)$ the maximum total degree of $F$, in
$z_j$'s and $w$.  From the degree triple of $P$ and the exposed
binomial factors, it is easy to realize that
\begin{align}
\label{eq.2165458}
D_w(F)
&\leq D(P) - (N-m)\ell
\ef,
\\
d(F) &\leq
d(P) - \ell
\ef,
\\
d_w(F) 
&
\leq d_m(P) - (N-m)\ell
\leq m\, d(P) - (N-m)\ell
\end{align}
(The inequalities come from the fact that cancellations may occur in
$P$ from the specialization.  The equation for $d_w(F)$ is obtained by
considering in $P$ the $m$-uple of variables $\{z_i\}_{i \in I}$.)
Furthermore, if $N \geq 2m$, $d_m(F)$ is defined, and we can also state
\be
d_m(F)
\leq d_m(P) - m \ell
\ef.
\ee
(This equation is obtained by considering in $P$ the $m$-uple of
variables $\{z_i\}_{i \in J}$ for some $J$ of size $m$ and disjoint
from $I$).

From the bounds above on the degree of $F$, and the fact that
$F^{(K)}_{N,m,\ell}(\vec{z}_{\ssm I},w)$ must satisfy the
$(m,\ell)$-wheel condition on $m$-uples of the $N-m$ remaining
variables $z_j$, we can prove in the various cases one of the
following
\begin{subequations}
\begin{gather}
D_w(F) \leq
D^*(N-m,m,\ell)
\ef;
\\
d_w(F) \leq 0
\textrm{\quad and\quad}
\big(\ 
d_m(F) \leq d^*_m(N-m,m,\ell)
\textrm{\quad or\quad}
d(F) = 0
\ \big)
\ef;
\\
f_{m,\ell}\big( D_w(F),d(F) \big)
\leq
f_{m,\ell}\big( D^*(N-m,m,\ell),d^*(N-m,m,\ell) \big)
\ef;
\end{gather}
\end{subequations}
and by induction on $N$ we conclude that
\be
\label{eq.45659254}
F^{(K)}_{N,m,\ell}(\vec{z}_{\ssm I},w)= 
c_K
\;
s_{N-m,m,\ell}(\vec{z}_{\ssm I})
\ef,
\ee
for some constant $c_K$.
However, $c_K$ cannot depend on $K$ either. This is seen by
specializing equations (\ref{eq.koreLM3}) and (\ref{eq.45659254}) to
$w=0$, which gives
\be
\label{eq.koreLM3bis}
P(\vec{z}_{\ssm I}, 0, \ldots, 0)
=
c_K
\;
s_{N-m,m,\ell}(\vec{z}_{\ssm I})
\prod_{j \in [N] \ssm I}
z_j^{\ell}
\ef.
\ee 

So, up to a multiplicative factor in $P$, we know that for any $I$ and
$K$ as above, the specialization to $z_{i_a}=q^{k_a}w$ of $P(\vec z)$
and of $s_{N,m,\ell}(\vec{z})$ are equal.  This is rephrased by saying
that the difference $R(\vec z):= s_{N,m,\ell}(\vec{z})-P(\vec z)$ is a
symmetric polynomial
satisfying the $(m-1,\ell+1)$-wheel condition, and furthermore
implies easily that $D(R)$, $d(R)$ and $d_m(R)$ are a triple of
entries smaller or equal to some triple $(D,d,d_m)$ satisfying (one of)
the degree condition under consideration (because $P$ does this by
hypothesis, and the Schur function does it explicitly, and the
difference can at most decrease the degrees through cancellations).

As all the degree conditions in our proposition are monotonic (in
particular, $f_{m,\ell}(D+\alpha,d+\beta) \geq f_{m,\ell}(D,d)$ if
$\alpha,\beta \geq 0$), the quantities in the conditions, as
functions of $D(R)$, $d(R)$ and $d_m(R)$, are bounded from above by
the analogous quantities as functions of $D^*$, $d^*$ and $d^*_m$ (for
parameters $(m,\ell)$).

Making an induction hypothesis in $m$, these degree bounds are to be
compared with the bounds for a symmetric function in $N$ variables,
satisfying the $(m-1,\ell+1)$-wheel condition, stated in the
proposition. Therefore, in our range of interest $m \geq 2$, $a \geq
1$, write $N=am+b=\tilde a(m-1)+\tilde b$, with $1\leq \tilde b\leq
m-1$. Clearly $\tilde a \geq a$. The triple entering the bounds to the
degrees of $R$ for the $(m,\ell)$ case reads
\begin{align}
D
&=
\ell
\big(
m \smbinom{a}{2} + a b
\big)
\ef;
\\
d
&=
\ell a
\ef;
\\
d_m
&=
\ell (N-m)
\ef;
\end{align}
while the triple entering the bounds for the $(m-1,\ell+1)$ case reads
\begin{align}
D'
&=
(\ell+1)
\big(
(m-1) \smbinom{\tilde{a}}{2} + \tilde{a} \tilde{b}
\big)
\ef;
\\
d'
&=
(\ell+1) \tilde{a}
\ef;
\\
d'_{m-1}
&=
(\ell+1) (N-m+1)
\ef.
\end{align}
In particular, $f_{m,\ell}(D,d) = \ell^2 a N/m$
and $f_{m-1,\ell+1}(D',d') = (\ell+1)^2 \tilde{a} N/(m-1)$.
As we have
\begin{subequations}
\begin{align}
d &< d'
\ef;
\\
d_{m-1} \leq d_m 
&< d'_{m-1}
\ef;
\\
f_{m-1,\ell+1}(D,d)
&<
f_{m-1,\ell+1}(D',d')
\ef;
\end{align}
\end{subequations}
(the last inequality comes with some algebra: the
difference is
$
f_{m-1,\ell+1}(D,d)
-
f_{m-1,\ell+1}(D',d')
=
-\frac{(\ell+1)N}{m-1}
\big( (\ell+1) \tilde{a} - \ell a \big)
-
\frac{\ell (\ell+m)}{m-1} \binom{a+1}{2}
$
and is negative at sight),
for any of the conditions in our list
we reach the conclusion that $R(\vec{z})=0$.
\end{proof}

\section*{Acknowledgements} 
\noindent
Part of the statements proven in this paper have been conjectured in
September 2009, when two of us (L.C.\ and A.S.) had the opportunity of
working together, within the programme \emph{StatComb09} at the
Institut H.~Poincar\'e -- Centre \'Emile Borel in Paris, that we thank
for support.

We thank Alain Lascoux for important discussions. In particular, at a
preliminary stage of this work, he suggested us the use of Bazin
Theorem for dealing with compound determinants.  Those conversations
had a crucial role in the development of our proof.





\begin{thebibliography}{99}


\bibitem{cit.B}
D.M.~Bressoud,
{\it Proofs and confirmations. The story of the alternating sign
  matrix conjecture.}
MAA Spectrum. Mathematical Association of America, Washington, DC;
Cambridge University Press, Cambridge, 1999.


\bibitem{cit.Colomo2p}
F.~Colomo and A.G.~Pronko,
{\it On two-point boundary correlations in the six-vertex model with
  domain wall boundary conditions,}
J.\ Stat.\ Mech.\ {\bf 2005} (2005) P05010\\
{\tt arXiv:math-ph/0503049}

\bibitem{cit.FJMM2}
B.~Feigin, M.~Jimbo, T.~Miwa, E.~Mukhin and Y.~Takeyama
{\it Symmetric polynomials vanishing on the diagonals shifted by roots
  of unity,}
Int.\ Math.\ Res.\ Not.\ {\bf 18} (2003) 999-1014\\
{\tt arXiv:math/0209126}

\bibitem{cit.FJMM}
B.~Feigin, M.~Jimbo, T.~Miwa and E.~Mukhin,
{\it Symmetric polynomials vanishing on the shifted diagonals and
  Macdonald polynomials,}
Int.\ Math.\ Res.\ Not.\ {\bf 18} (2003) 1015--1034\\
{\tt arXiv:math/0209042}

\bibitem{cit.FZJ}
T.~Fonseca and P.~Zinn-Justin,
{\it On the doubly refined enumeration of alternating sign matrices
  and totally symmetric self-complementary plane partitions,}
Electron.\ J.\ Combin.\ {\bf 15} R81 (2008)
{\tt arXiv:0803.1595}

\bibitem{izergin}
A.G.~Izergin, 
{\it Partition function of the six-vertex model in a finite volume,}
Sov.\ Phys.\ Dokl.\ {\bf 32} (1987), 878.




\bibitem{korepin}
V.E.~Korepin, 
{\it Calculation of norms of Bethe wave functions,}
Commun.\ Math.\ Phys.\ {\bf 86} (1982), 391.

\bibitem{cit.Kdet}
C.~Krattenthaler,
{\it Advanced Determinant Calculus,}
S\'em.\ Lothar.\ Combin.\ {\bf 42} ``The Andrews Festschrift'' 
art.~B42q (1999) 


\bibitem{cit.K}
G.~Kuperberg,
{\it Another proof of the alternating-sign matrix conjecture,}
Int.\ Math.\ Res.\ Not.\ {\bf 3} 139--150 (1996)
{\tt arXiv:math/9712207}

\bibitem{macdonald}
I.G.~Macdonald, {\it Symmetric Functions and Hall Polynomials.}
Oxford Mathematical Monographs. Oxford University Press (1999). 


\bibitem{cit.Lpf}
A.~Lascoux,
{\it Pfaffians and representations of the symmetric group,}
Acta Math.\ Sinica (Eng.) {\bf 25} 1929-1950 (2009)
{\tt arXiv:math/0610510}


\bibitem{cit.MRR-asm}
W.H.~Mills, D.P.~Robbins and H.~Rumsey,
{\it  Alternating-sign matrices and descending plane
partitions,}
J.\ Combin.\ Theory Ser.\ A {\bf 34} 340--359 (1983)

\bibitem{cit.muir}
T.\ Muir,
{\it A Treatise on the Theory of Determinants,}
Courier Dover Publ., 2003, reprint of the 1939 edition
redacted by W.H.\ Metzler.

\bibitem{cit.okada}
S.~Okada,
{\it Enumeration of symmetry classes of alternating sign matrices and
  characters of classical groups,}
J.\ of Alg.\ Combin.\ {\bf 23}(1) (2006) 43-69  
{\tt arXiv:math/0408234}



\bibitem{stanley}
R.P.~Stanley {\it Enumerative Combinatorics, Volume 2,}
Cambridge University Press, Cambridge, 2001.

\bibitem{cit.strog}
Yu.~Stroganov,
{\it Izergin-Korepin determinant at a third root of unity,}
Theor.\ Math.\ Phys.\ {\bf 146} 65-76 (2004) 
[russian: Teor.\ Mat.\ Fiz.\ {\bf 146} 53-62 (2004)]
{\tt arXiv:math-ph/0204042}

\bibitem{cit.Z}
D.~Zeilberger,
{\it Proof of the alternating sign matrix conjecture,}
The Foata Festschrift -- Electron.\ J.\ Combin.\ {\bf 3} R13 (1996)


\bibitem{cit.pzj_hs}
P.~Zinn-Justin,
{\it Combinatorial Point for Fused Loop Models,}
Comm.\ Math.\ Phys.\ {\bf 272} 
661-682 (2007)
{\tt arXiv:math-ph/0603018}

\end{thebibliography}
\end{document}